\documentclass[11pt,reqno]{amsart}
\usepackage{array}
\usepackage{hyperref}
\usepackage{geometry}
\usepackage{enumitem}
\usepackage{amsmath}
\usepackage{amssymb}
\usepackage{amsthm}
\usepackage{amsrefs}
\usepackage{amsfonts}
\usepackage[dvipsnames]{xcolor}
\usepackage{mathtools}
\usepackage{stackengine}
\usepackage{verbatim}
\usepackage{tikz-cd} 
\usepackage{mathrsfs}
\usepackage{float}
\makeatletter
\@namedef{subjclassname@2020}{%
  \textup{2020} MSC}
\makeatother

\hypersetup{
colorlinks   = true, 
urlcolor     = blue, 
linkcolor    = purple, 
citecolor   = blue 
}

\def\XXint#1#2#3{{\setbox0=\hbox{$#1{#2#3}{\int}$ }
\vcenter{\hbox{$#2#3$ }}\kern-.6\wd0}}

\makeatletter
\newcommand*{\rom}[1]{\expandafter\@slowromancap\romannumeral #1@}

\newcommand{\diam}{\operatorname{diam}}

\newcommand{\SL}{\mathrm{SL}}
\newcommand{\sspan}{\operatorname{span}}

\newcommand{\M}{\mathrm{M}}

\newcommand{\X}{\mathcal{X}}

\newcommand{\R}{\mathbb{R}}

\newcommand{\e}{\varepsilon}

\newcommand{\Z}{\mathbb{Z}}

\newcommand{\SO}{{\mathrm{SO}}}

\newcommand{\te}{{\tilde{e}}}

\newcommand{\N}{\mathbb{N}}

\newcommand{\bthm}{\begin{thm}}
\newcommand{\ethm}{\end{thm}}
\newcommand{\bproof}{\begin{proof}}
\newcommand{\eproof}{\end{proof}}
\newcommand{\blem}{\begin{lem}}
\newcommand{\elem}{\end{lem}}
\newcommand{\brem}{\begin{rem}}
\newcommand{\erem}{\end{rem}}
\newcommand{\eeqn}{\end{equation}}
\newcommand{\eeqnn}{\end{equation*}}
\newcommand{\beqn}{\begin{equation}}
\newcommand{\beqnn}{\begin{equation*}}
\newcommand{\eprop}{\end{prop}}
\newcommand{\eexm}{\end{exm}}
\newcommand{\enexm}{\end{nexm}}
\newcommand{\ecor}{\end{cor}}
\newcommand{\bcor}{\begin{cor}}
\newcommand{\bexm}{\begin{exm}}
\newcommand{\bnexm}{\begin{nexm}}
\newcommand{\bprop}{\begin{prop}}
\newcommand{\bdefn}{\begin{defn}}
\newcommand{\edefn}{\end{defn}}
\newcommand{\benum}{\begin{enumerate}}
\newcommand{\eenum}{\end{enumerate}}

\newcommand{\bfe}{\mathbf{e}}

\newcommand{\Mat}{\M_{m \times n}(\R)}
\newcommand{\Prob}{\mathrm{Prob}}
\newcommand{\supp}{\operatorname{supp}}
\newcommand{\Sing}{\mathrm{Sing}}
\newcommand{\VSing}{\mathrm{Sing}}

\newcommand{\mrm}{\mathrm}
\newcommand{\mc}{\mathcal}
\newcommand{\Fcal}{\mc{F}}
\newcommand{\Kcal}{\mc{K}}
\newcommand{\Pcal}{\mc{P}}
\newcommand{\tPcal}{\Tilde{\mc{P}}}

\renewcommand{\b}{\beta}

\newcommand{\bfn}{\hat{\eta}}
\newcommand{\an}{{\alpha_{\mathbf{\eta}}} }
\newcommand{\mur}{\mu^{(r)}}
\newcommand{\bfr}{\mathbf{r}}
\newcommand{\bfs}{\mathbf{s}}
\newcommand{\bfv}{\mathbf{v}}
\newcommand{\Tick}{\checkmark}
\newcommand{\Div}{\mathrm{Div}}
\newcommand{\Divergent}{\mathrm{Divergent}_{a,b}}

\title{On the packing dimension of weighted singular matrices on fractals}

\begin{document}
\theoremstyle{plain}
\newtheorem{thm}{Theorem}[section]
\newtheorem{lem}[thm]{Lemma}
\newtheorem{prop}[thm]{Proposition}
\newtheorem{cor}[thm]{Corollary}
\newtheorem{question}{Question}
\newtheorem{con}{Conjecture}
\theoremstyle{definition}
\newtheorem{defn}[thm]{Definition}
\newtheorem{exm}[thm]{Example}
\newtheorem{nexm}[thm]{Non Example}
\newtheorem{prob}[thm]{Problem}

\theoremstyle{remark}
\newtheorem{rem}[thm]{Remark}

\author{Gaurav Aggarwal}
\address{\textbf{Gaurav Aggarwal} \\
School of Mathematics,
Tata Institute of Fundamental Research, Mumbai, India 400005}
\email{gaurav@math.tifr.res.in}

\author{Anish Ghosh}
\address{\textbf{Anish Ghosh} \\
School of Mathematics,
Tata Institute of Fundamental Research, Mumbai, India 400005}
\email{ghosh@math.tifr.res.in}

\date{}

\thanks{ A. Ghosh gratefully acknowledges support from a J. C. Bose grant and an endowment from the Infosys foundation. G. Aggarwal and A. Ghosh gratefully acknowledges a grant from the Department of Atomic Energy, Government of India, under project $12-R\&D-TFR-5.01-0500$. }



\begin{abstract}  
We provide the first known upper bounds for the packing dimension of weighted singular and weighted $\omega$-singular matrices. We also prove upper bounds for these sets when intersected with fractal subsets. The latter results, even in the unweighted setting, are already new for matrices. Further, even for row vectors, our results enlarge the class of fractals for which bounds are currently known. We use methods from homogeneous dynamics, in particular we provide upper bounds for the packing dimension of points on the space of unimodular lattices, whose orbits under diagonal flows $p$-escape on average.

\end{abstract}

\maketitle

\tableofcontents

\section{Introduction}


    
Singular matrices are a class of vectors for which Dirichlet's theorem in Diophantine approximation can be infinitely improved. In this paper, we will be concerned with a more general setting, that of weighted singular matrices, to study which we need to introduce quasi-norms. Fix $m,n \in \N$, $a= (a_1, \ldots, a_m) \in \R^m$ and $b=(b_1, \ldots, b_n)$ such that
    \begin{align}
        a_1 \geq a_2 \geq \cdots \geq a_m>0, &\quad b_1 \geq b_2 \geq \cdots \geq b_n>0, \label{eq: a b 1}\\
        a_1 + \cdots + a_m = 1, &\quad b_1 + \cdots + b_n =1. \label{eq: a b 2}
    \end{align}
    We define a \emph{quasi-norm} $\|.\|_a$ on $\R^m$ as $\|x\|_a= \max_i |x_i|^{1/a_i}$ for all $x= (x_1, \ldots, x_m) \in \R^m$. Similarly, define a quasi-norm $\|.\|_b$ on $\R^n$ as $\|y\|_b= \max_j |y_j|^{1/b_j}$ for all $y= (y_1, \ldots, y_n) \in \R^n$.
\begin{defn}
     A matrix $\theta \in \Mat$ is called {\bf $(a,b)$-singular} if for every $\e>0$, there exists $T_{\e}>0$ such that for all $T>T_{\e}$, there exists $(p,q) \in \Z^m \times (\Z^n \setminus \{0\})$ such that the following holds
     \begin{align*}
        \|p +  \theta q\|_{a} &\leq \frac{\e}{T},  \\
         \|q\|_b &\leq T.
    \end{align*}
     We denote by $\Sing(a,b) \subset \Mat$, the set of all $(a,b)$-singular vectors in $\Mat$. 
\end{defn}

\begin{defn}
We define the $(a,b)$-uniform exponent of irrationality of a matrix $\theta \in \Mat$, denoted $\omega(\theta,a,b)$, as the supremum of $\omega>0$ such that for all large $T$, there exists $(p,q) \in \Z^m \times (\Z^n \setminus \{0\})$ such that the following holds
     \begin{align*}
        \|p +  \theta q\|_{a} &\leq \frac{1}{T^{1+ \omega}}  \\
         \|q\|_b &\leq T.
    \end{align*}
We call a matrix $\theta \in \Mat$, an {\bf $(a,b, \omega)$-singular} matrix (we say $\theta \in \VSing(a,b,\omega)$) if $\omega(\theta,a,b)\geq \omega$.

 We refer the reader to \cite{BL} for a comprehensive study of uniform Diophantine exponents. In the unweighted setting, $\omega$-singular matrices have been considered previously, cf. \cites{BugeaudCheungChevallier, DFSU, schleischitz2022}. From Minkowski's convex body Theorem, it is easy to see that $\VSing(a,b,0)= \Mat$. Moreover, if ${\omega}>0$, then $\VSing(a,b,\omega) \subset \Sing(a,b)$. 
\end{defn}

 Although the study of the classical `non-weighted' situation is more ubiquitous, Diophantine approximation with weights has been extensively studied in recent times. We refer the reader to \cite{CGGMS} for an introduction. The set of singular vectors is clearly nonempty, and due to a classical result of Khintchine, it contains uncountably many vectors when the dimension is greater than $1$. Moreover, a modification of a classical argument in \cite{Casselsbook} shows that the set of singular vectors has zero Lebesgue measure.\\

The question of estimating the Hausdorff dimension of singular vectors turns out to be a difficult problem and has received much attention in the last two decades. In a landmark work \cite{Cheung}, Y. Cheung showed that the dimension of $\Sing((1/2,1/2), 1) \subset \R^2$ is $4/3$. This was subsequently generalized to $\R^n$ in an important work of Cheung and Chevallier \cite{CheungChevallier}. Another important result was obtained by Kadyrov, Kleinbock, Lindenstrauss and Margulis in \cite{KKLM} using methods from homogeneous dynamics. Namely, a sharp upper bound on the more general set of \emph{singular on average} $m \times n$ matrices was obtained using integral inequalities as introduced in the famous work \cite{EMM98} on the Oppenheim conjecture. The complementary lower bound was obtained by Das, Fishman, Simmons and Urba\'nski in \cite{DFSU} using methods from the parametric geometry of numbers, see also the recent paper \cite{Solan} for more results in this direction. In \cite{ShahYang}, Shah and Yang obtained bounds for the dimension of certain singular vectors lying on affine subspaces.\\

Following a research direction suggested by Mahler, the topic of metric Diophantine approximation on fractals has recently witnessed a surge in popularity. In \cite{BugeaudCheungChevallier}, the authors (Problem 6) pose the following question: ``What is the Hausdorff dimension of the set of singular pairs whose entries belong to the middle third Cantor set?"
Clearly, this question can be framed in much wider generality. In \cite{Khalilsing}, Khalil upgraded the integral inequality approach of \cite{KKLM} using a beautiful argument to deal with singular vectors on fractals. That is, an upper bound on the Hausdorff dimension of singular vectors lying on self-similar fractals in $\R^m$ that satisfied the open set condition was obtained. The only other progress that we are aware of is the work of Schleischitz \cite{schleischitz2022} where \emph{lower bounds} on the packing dimension of singular vectors lying on a certain specific class of fractals were obtained.\\

As for weighted singular vectors, in \cite{LSST}, the Hausdorff dimension of weighted singular vectors in $\R^2$ was calculated. In higher dimensions, a lower bound was recently obtained by Kim and Park in \cite{KimPark2024} using an appropriate generalization of the technique of \cite{LSST} that involves the construction of a fractal set contained in $\Sing(a, 1)$. In general, this question is wide open and of considerable interest as pointed out, for instance in Section 5.7 of \cite{DFSU}.\\ 

In this paper, we provide a general method for obtaining upper bounds on the packing dimension of singular and, more generally, $(a,b,\omega)$ singular vectors in a wide variety of contexts. We consider matrices and also permit intersections with fractals. Even in the unweighted setting, we expand the current known scope of Problem 6 of Bugeaud, Cheung and Chevalier mentioned above, by including more fractals.  We can now state our main theorem.

\begin{thm}
    \label{main thm}
    Fix $m, n \in \N$, and let $a = (a_1, \ldots, a_m) \in \R^m$ and $b = (b_1, \ldots, b_n) \in \R^n$ satisfy conditions \eqref{eq: a b 1} and \eqref{eq: a b 2}. For $1 \leq l \leq d-1$, define $w_l$ as follows:
    $$
    w_l = 
    \begin{cases} 
        a_m + \cdots + a_{m+1-l}, & \text{if } l \leq m, \\
        1 - (b_1 + \cdots + b_{l-m}), & \text{if } l > m. 
    \end{cases}
    $$

    For each $1 \leq i \leq m$ and $1 \leq j \leq n$, let $\Phi_{ij}$ be an iterated function system (IFS) consisting of contracting similarities on $\R$ with equal contraction ratios, satisfying the open set condition. Let $\Kcal_{ij}$ be the limit set of $\Phi_{ij}$, and define 
    $$
    \Kcal = \{\theta \in \Mat : \theta_{ij} \in \Kcal_{ij} \text{ for all } i, j\}.
    $$
    Assume that $\dim_H(\Kcal_{ij})>0$ for all $i, j$. Then there exist constants $\eta_1, \ldots, \eta_{m+n-1} > 0$ (depending only on $\Kcal$) such that the following results hold:
    \begin{itemize}
        \item The packing dimension of $\Sing(a, b) \cap \Kcal$ satisfies:
        \begin{align}
        \label{eq: main thm 1}
            \dim_P(\Sing(a, b) \cap \Kcal) \leq \dim_P(\Kcal) - \frac{1}{a_1 + b_1} \left( \min_{1 \leq l \leq d-1} \eta_l w_l \right).
        \end{align}
        
        \item For any $\omega > 0$, the packing dimension of $\VSing(a, b, \omega) \cap \Kcal$ satisfies:
        \begin{align}
        \label{eq: main thm 2}
        \dim_P(\VSing(a, b, \omega) \cap \Kcal) \leq \dim_P(\Kcal) - \frac{1}{a_1 + b_1} \left( \min_{1 \leq l \leq d-1} \eta_l w_l + \frac{\eta_1 a_m b_n \omega}{a_m + b_n + a_m \omega} \right).
        \end{align}
    \end{itemize}

    Moreover, the constants $\eta_1, \ldots, \eta_{m+n-1}$ can be explicitly chosen in the following cases:
    \begin{enumerate}
        \item If $\Kcal = \M_{m \times n}([0,1])$, we can take:
        \begin{align}
        \eta_l = 
        \begin{cases}
            \frac{m}{l}, & \text{if } l \leq m, \\
            \frac{n}{m+n-l}, & \text{if } l > m.
        \end{cases}
        \end{align}

        \item If $n = 1$, we can take:
       \begin{align}
        \eta_l = \frac{m}{l} \min_{1 \leq i \leq m} \dim_H(\Kcal_{i1}).
       \end{align}

        \item If $m = 1$, we can take:
        \begin{align}
        \eta_l = \frac{n}{n+1-l} \min_{1 \leq j \leq n} \dim_H(\Kcal_{1j}).
       \end{align}
    \end{enumerate}
\end{thm}

\begin{rem}
    Note that the packing dimension of any set is larger than the Hausdorff dimension of the set. Thus \eqref{eq: main thm 1} and \eqref{eq: main thm 2} also give an upper bound for the Hausdorff dimension of the respective sets.
\end{rem}

\begin{rem}
    As will be shown in the proof, the constants $\eta_1, \ldots, \eta_{m+n-1} \in \R$ in Theorem \ref{main thm}, can be chosen as any point in the following closed region:
    \begin{align*}
        0<\eta_i &\leq \zeta_i(\mu) \quad \text{for all } 1 \leq i \leq d-1, \\
    \frac{1}{\eta_{i-j}} + \frac{1}{\eta_{i+j}} &\leq \frac{2}{\eta_i} \text{ for all } 1 \leq i \leq m+n-1, j \leq \min\{i, m+n-i\},
    \end{align*}
    where $1/\eta_0 = 1/\eta_{m+n} := 0$ and $\mu$ equals product of normalised $\dim(\Kcal_{ij})$-dimensional Hausdorff measure on $\Kcal_{ij}$. The definitions of $\zeta_1(\mu), \ldots, \zeta_{m+n-1}(\mu)$ are given in Section \ref{sec: Critical Exponent}. These constants encode information about whether the $\Kcal$ is close to certain hypersurfaces. For $m = 1$ or $n=1$, these hypersurfaces degenerate to coordinate planes, enabling non-trivial lower bounds for $\zeta_1(\mu), \ldots, \zeta_{d-1}(\mu)$, see Section \ref{sec: Critical Exponent: Special Cases}. Similarly, explicit computations are possible when $\Kcal = \M_{m \times n}([0,1])$.
\end{rem}

One of the main ingredients of the proof of Theorem \ref{main thm} is Dani’s correspondence, which gives that $(a,b)$-singular matrices corresponds to certain divergent trajectories in the space $\X_{m+n}$ of unimodular lattices in $\R^{m+n}$. More precisely, let 
\begin{align}
    g_t = \begin{pmatrix}
        t^{a_1} \\ & \ddots \\ &&t^{a_m} \\ &&& t^{-b_1} \\ &&&& \ddots \\ &&&&&t^{-b_n} \\
    \end{pmatrix}, \quad
    u(\theta)= \begin{pmatrix}
        I_m & \theta \\ & I_n
    \end{pmatrix}.
    \end{align}
for $t>0$ and $\theta \in \Mat$. Then $\theta$ is $(a,b)$-singular if and only if the diagonal orbit $(g_tu(\theta)\Z^d)_{t \geq 1}$ is divergent. A different way of quantifying the notion of singularity is the notion of singularity on average introduced in \cite{KKLM}. Given $x \in \X_{m+n}$ and $0< p \leq 1$, we define $\Divergent(x,p)$ as the set of all $\theta \in \Mat$ such that
$$
\lim_{\e \rightarrow 0} \liminf_{T \rightarrow \infty} \frac{1}{T} \int_{0}^T \delta_{g_{e^t}u(\theta) x} \{y \in \X: \lambda_1(y) \leq \e \} \, dt \geq p,
$$
where $\lambda_1(y)$ denotes the length of the shortest non-zero vector in $y$. It is then clear and will also be shown in Lemma~\ref{lem: Sing Dynamical Interpretation} that $\Sing(a,b) \subset \Divergent(\Z^{m+n}, 1)$. Similarly, as will be shown in Lemma \ref{lem: omega sing dynamical Interpretation}, if $\theta$ is $(a,b,\omega)$-singular then for all $\omega'< \omega$, there exists $T(\omega')$ such that for all $t>T(\omega')$, we have $$\lambda_1(g_tu(\theta)\Z^d) <t^{\frac{ -a_mb_n\omega'}{a_m+b_n+ a_m\omega'}}$$ for all large $t$. Thus, Theorem \ref{main thm} follows from the following more general theorem.

\begin{thm}
    \label{main thm 2}
    With notation as in Theorem \ref{main thm}, let $\Kcal_{ij}$ be the limit set of $\Phi_{ij}$, and define 
    $$
    \Kcal = \{\theta \in \Mat : \theta_{ij} \in \Kcal_{ij} \text{ for all } i, j\}.
    $$
    Assume that $\dim_H(\Kcal_{ij})>0$ for all $i, j$. Then there exist constants $\eta_1, \ldots, \eta_{m+n-1} > 0$ (depending only on $\Kcal$) such that the following results hold for any $x \in \X_{m+n}$ and $0< p \leq 1$:
    \begin{itemize}
        \item The packing dimension of $\Divergent(x,p) \cap \Kcal$ is less than or equal to
        \begin{align}
        \label{eq: main thm 2 1}
             \dim_P(\Kcal) - \frac{p}{a_1 + b_1} \left( \min_{1 \leq l \leq d-1} \eta_l w_l \right).
        \end{align}
        
        \item For any $\gamma > 0$, the packing dimension of the set $\{\theta \in \Kcal: \text{ for all large $t$, we have } \lambda_1(g_tu(\theta)x)\leq t^{-\gamma} \}$ is less than or equal to
        \begin{align}
        \label{eq: main thm 2 2}
         \dim_P(\Kcal) - \frac{1}{a_1 + b_1} \left( \min_{1 \leq l \leq d-1} \eta_l w_l + \eta_1 \gamma \right).
        \end{align}
    \end{itemize}

    Moreover, the constants $\eta_1, \ldots, \eta_{m+n-1}$ can be explicitly chosen in the following cases:
    \begin{enumerate}
        \item If $\Kcal = \M_{m \times n}([0,1])$, we can take:
        \begin{align}
        \label{eq: main thm 3}
        \eta_l = 
        \begin{cases}
            \frac{m}{l}, & \text{if } l \leq m, \\
            \frac{n}{m+n-l}, & \text{if } l > m.
        \end{cases}
        \end{align}

        \item If $n = 1$, we can take:
       \begin{align}
        \label{eq: main thm 4}
        \eta_l = \frac{m}{l} \min_{1 \leq i \leq m} \dim_H(\Kcal_{i1}).
       \end{align}

        \item If $m = 1$, we can take:
        \begin{align}
        \label{eq: main thm 5}
        \eta_l = \frac{n}{n+1-l} \min_{1 \leq j \leq n} \dim_H(\Kcal_{1j}).
       \end{align}
    \end{enumerate}
\end{thm}

The following table summarizes the existing literature (``Existing") and the results we prove in this paper (``New").
\begin{table}[H]
\centering
\renewcommand{\arraystretch}{1.5} 
\setlength{\tabcolsep}{8pt}     
\begin{tabular}{|>{\centering\arraybackslash}m{4cm}|>{\centering\arraybackslash}m{4cm}|>{\centering\arraybackslash}m{6cm}|}
\hline
\textbf{Results} & \textbf{Existing} & \textbf{New} \\ 
\hline
$\mathrm{Sing}$ & \cites{Cheung, CheungChevallier, KKLM, DFSU} & Agree \\ 
\hline
$\mathrm{Divergent}(\Z^d,p)$ & \cites{KKLM, DFSU} & Agree \\ 
\hline
$\mathrm{Sing}( \omega)$ & \cites{BugeaudCheungChevallier,DFSU} & Weak \\ 
\hline
$\mathrm{Sing} \cap \mathrm{Fractal}$ &  \begin{itemize}
    \item Upper bound known for $m=1$ by \cite{Khalilsing}
    \item Allowed class of fractals = self-similar fractals in $\R^n$ satisfying the open set condition
\end{itemize}  & \begin{itemize}
    \item Upper bound for general $m,n$
    \item Allowed class of fractals= product of self-similar fractals in $\R$ with equal contraction ratio, satisfying the open set condition
    \item Upper bound agree for all common examples
    \item New example of fractal for $m=1$: Product of Middle third Cantor set with Middle fifth Cantor set
\end{itemize} \\ 
\hline
$\mathrm{Divergent}(\Z^d,p) \cap \mathrm{Fractal}$ & $\times$ & \Tick \\ 
\hline

$\VSing(\omega) \cap \mathrm{Fractal}$ & Lower bound for $n=1$ is known for certain class of fractals by \cite{schleischitz2022} & Upper bound for general $m,n$ \\
\hline
\end{tabular}
\caption{Comparison of Existing and New Results with equal weight assumption, i.e., $a= \left(\frac{1}{m}, \ldots, \frac{1}{m} \right)$ and $b= \left(\frac{1}{n}, \ldots, \frac{1}{n} \right)$}
\label{tab:comparison 1 }
\end{table}

\begin{table}[H]
\centering
\renewcommand{\arraystretch}{1.5} 
\setlength{\tabcolsep}{8pt}     
\begin{tabular}{|>{\centering\arraybackslash}m{4cm}|>{\centering\arraybackslash}m{4cm}|>{\centering\arraybackslash}m{6cm}|}
\hline
\textbf{Results} & \textbf{Existing} & \textbf{New} \\ 
\hline
$\mathrm{Sing}(a, b)$ & \begin{itemize}
    \item Hausdorff dimension known for $m=2$, $n=1$ by \cite{LSST}
    \item Lower bound for Hausdorff dimension for $n=1$ by \cite{KimPark2024}
\end{itemize} & Upper bound on packing dimension for general $m,n$ \\ 
\hline
$\Divergent(\Z^d,p)$ & $\times$ & \Tick \\ 
\hline
$\mathrm{Sing}(a, b, \omega)$ & $\times$ & \Tick \\ 
\hline
$\mathrm{Sing}(a, b) \cap \mathrm{Fractal}$ & $\times$  & \Tick \\ 
\hline
$\Divergent(\Z^d,p) \cap \mathrm{Fractal}$ & $\times$ & \Tick \\ 
\hline
$\VSing(a,b,\omega) \cap \mathrm{Fractal}$ & $\times $ & \Tick \\
\hline
\end{tabular}
\caption{Comparison of Existing and New Results with unequal assumption, i.e., $a\neq \left(\frac{1}{m}, \ldots, \frac{1}{m} \right)$ or $b \neq \left(\frac{1}{n}, \ldots, \frac{1}{n} \right)$}
\label{tab:comparison 2}
\end{table}

\begin{rem}
The proof of Theorem \ref{main thm 2} closely follows the framework developed in \cite{KKLM}. However, the setting of fractal measures and weighted dynamics introduces several additional challenges. To explain these, we divide the proof into two main parts.

The first part involves constructing a family of height functions on the space of unimodular lattices, tailored to satisfy the contraction hypothesis for fractal measures. This construction relies on establishing an expansion property for representations and using the integral inequalities introduced in \cite{EMM98}. Proving the expansion property for fractal measures constitutes the first key contribution of the paper.

The second part applies the contraction hypothesis for the height function iteratively to obtain an upper bound on exceptional trajectories. Although this approach is inspired by \cite{KKLM}, the setting with unequal weights presents significant technical difficulties. The unequal weights result in uneven expansion on $\Mat$ under conjugation, making the iterative process particularly delicate, especially when dealing with fractal measures. To address these difficulties, we prove the contraction hypothesis not only for the original fractal measure but also for measures arising from small perturbations. Furthermore, the arguments require careful handling of heavy notation and intricate calculations to ensure the validity of the iterative process. These iterative arguments represent the second key contribution of this paper.
\end{rem}

\subsection{Acknowledgements}
The authors are extremely grateful to the anonymous referee for a very careful reading of the manuscript and for pointing out several inaccuracies. The referee’s helpful comments have significantly improved the clarity of the paper.

\section{Notation}
The following notation will be used throughout the paper.

\subsection{Packing Dimension}
The $i$-dimensional packing measure of a set $F \subset \R^l$ is defined as 
\begin{align}
    \label{eq: def P}
    \Pcal^i(F) := \inf \left\{\sum_{j= 1}^\infty \tPcal^i(F_i):  F \subset \bigcup_{j=1}^\infty  F_j \right\}, 
\end{align}
where 
\begin{align}
    \label{eq: def til P}
    \tPcal^i(F) = \inf_{\e>0} \sup \left\{ \sum_{j=1}^\infty |B_j|^i : \substack{(B_j)_1^\infty  \text{ is a countable disjoint collection of balls } \\ \text{ with centers in $F$ and with diameter $|B_j| < \e$ for all $j$.}} \right\}
\end{align}
The packing dimension of a set $F \subset \R^l$ is defined as
\begin{align}
    \label{eq: def dim P}
    \dim_P(F)= \inf\{i: \Pcal^i(F)=0\}= \sup\{i:  \Pcal^i(F)= \infty\}.
\end{align}

We will need the following important lemma.
\begin{lem}[{\cite[Lem.~3.8]{Falconer}}]
\label{lem: Falconer}
    Let $F$ be a non-empty bounded subset of $\R^l$. Then
    $$
    \dim_P(F) \leq \limsup_{\delta \rightarrow 0} \frac{\log L_{\delta}(F)}{-\log \delta},
    $$
    where $L_\delta(F)$ denotes the smallest number of sets of diameter at most $\delta$ that cover $F$.
\end{lem}

\subsection{Homogeneous spaces}
 Fix $m,n \in \N$ and set $d=m+n$. Fix $a= (a_1, \ldots, a_m) \in \R^m$ and $b=(b_1, \ldots, b_n)$ such that
    \begin{align*}
        a_1 \geq a_2 \geq \cdots \geq a_m>0, &\quad b_1 \geq b_2 \geq \cdots \geq b_n>0, \\
        a_1 + \cdots + a_m = 1, &\quad b_1 + \cdots + b_n =1. 
    \end{align*}
    
Let $G=\SL_{d}(\R), \Gamma = \SL_{d}(\Z)$ and $\X= G/\Gamma$. The space $\X$ can be naturally identified with the space of unimodular lattices in $\R^{d}$, via the identification $A\Gamma \mapsto A\Z^d$.\\

 For $t>0$ and $\theta \in \Mat$, define
\begin{align}
    g_t = \begin{pmatrix}
        t^{a_1} \\ & \ddots \\ &&t^{a_m} \\ &&& t^{-b_1} \\ &&&& \ddots \\ &&&&&t^{-b_n} \\
    \end{pmatrix}, \quad
        \label{eq: def u theta}
    u(\theta)= \begin{pmatrix}
        I_m & \theta \\ & I_n
    \end{pmatrix}.
    \end{align}

\subsection{Iterated Function Systems}

    A contracting similarity is a map $\R \rightarrow \R$ of the form $x \mapsto cx+ y$ where $c \in (0,1)$, and $y \in \R$. A \emph{finite similarity Iterated Function System} with constant ratio ({\em IFS}) on $\R$ is a collection of contracting similarities $\Phi= (\phi_e: \R \rightarrow \R)_{e \in E}$ indexed by a finite set $E$, called the alphabet, such that there exists a constant $c \in (0,1)$ independent of $e$ so that
$$
\phi_e(x)= cx + w_e,
$$
for all $e \in E$.

Let $B = E^\N$. The coding map of an IFS $\Phi$ is the map $\sigma: B \rightarrow \R$ defined by the formula
\begin{align}
\label{eq:def sigma}
    \sigma(b)= \lim_{l \rightarrow \infty} \phi_{b_1} \circ \cdots \circ \phi_{b_l}(0).
\end{align}
It is well known that the limit in $\eqref{eq:def sigma}$ exists and that the coding map is continuous. The image of $B$ under the coding map, called the limit set of $\Phi$, is a compact subset of $\R$, which we denote by $\Kcal= \Kcal(\Phi)$. We define for $\te=(e_1, \ldots, e_l) \in E^l$, 
\begin{align}
    \label{eq: def K w, collection}
    \Kcal_\te= \phi_{e_1} \circ \cdots \circ \phi_{e_l}(\Kcal), \quad \text{ and set } \Fcal(l)= \{\Kcal_{\te} : \te \in E^l\}.
\end{align}

We will say that $\Phi$ satisfies the \emph{open set condition} (OSC for short) if there exists a non-empty open subset  $U \subset \R$ such that the following holds
    \begin{align*}
        \phi_e(U) \subset U &\text{ for every } e \in E \\
        \phi_e(U) \cap \phi_{e'}(U) = \emptyset, &\text{ for every } e \neq e' \in E.
    \end{align*}

Let $\Prob(E)$ denote the space of probability measures on $E$. For each $\nu \in \Prob(E)$ we can consider the measure $\sigma_* \nu^{\otimes \N}$ under the coding map. A measure of the form $\sigma_* \nu^{\otimes \N}$ is called a \emph{Bernoulli measure}. 

The following proposition is well known (see for e.g. \cite[Prop.~5.1(4), Thm.~5.3(1)]{Hutchinson} for a proof).

\begin{prop}
    \label{prop: important IFS}
    Suppose $\Phi= \{\phi_e: e \in E\}$ is an IFS satisfying the open set condition with the limit set $\Kcal$. Let $c$ denote the common contraction ratio of $(\phi_e)_{e \in E}$ and $p= \#E$. Then the Hausdorff dimension $s$ of $\Kcal$ is $-\log p/ \log c$. Also, the $s$-dimensional Hausdorff measure $H^s$ satisfies $0< H^s(\Kcal)< \infty$. Moreover if $\mu$ denotes the normalised restriction of $H^s$ to $\Kcal$, then $\mu$ is a Bernoulli measure and equals $\sigma_* \nu^{\otimes \N}$, where $\nu$ is the uniform measure on $E$, i.e, $\nu(F):= \#F/\#E$ for all $F \subset E$. Additionally, for every $l \in \N$ and distinct sequences $\te_1 \neq \te_2 \in E^l$, we have $\mu(\Kcal_{\te_1} \cap \Kcal_{\te_2}) = 0$. Furthermore, there exists a constant $\lambda >0$ such that for all $x \in \R$ and $y>0$, we have
    \begin{align}
        \label{eq: imp IFS}
        \mu([x-y, x+y]) \leq \lambda y^s.
    \end{align}
\end{prop}

For the rest of the paper, we fix for all $1 \leq i \leq m$ and $1 \leq j \leq n$, an IFS $\Phi_{ij}=\{\phi_{ij,e}: e\in E_{ij}\}$ satisfying the open set condition, with common contraction ratio $c_{ij}$ and limit set $\Kcal_{ij}$. Let $p_{ij}= \#E_{ij}$ and $s_{ij}= -\log p_{ij}/ \log c_{ij} >0$. Let us define $\Kcal =\{\theta \in \Mat : \theta_{ij} \in \Kcal_{ij} \} $ and $s= \sum_{ij} s_{ij}$. Let $\mu_{ij}$ denote the normalised restriction of $H^{s_{ij}}$ to $\Kcal_{ij}$ and define the measure 
\begin{align}
    \label{eq: def mu}
    \mu= \otimes_{ij} \mu_{ij}
\end{align}
on $\Kcal$. 

Let $\Xi \subset \Mat$ be defined as $\Xi=\{r \in \Mat: r_{ij} \in [c_{ij},c_{ij}^{-1}]\}$. For all $1 \leq i \leq m$ and $1 \leq j \leq n$, we define $\mur_{ij}$ as the measure on $\R$ obtained by pushing forward the measure $\mu_{ij}$ under map $ x \mapsto r_{ij}x$. We also define $\mur = \prod_{ij}\mur_{ij}$, viewed as a measure on $\Mat$.

\subsection{Representation Theory}
\label{subsec: Rep Theory}
For all $1 \leq l \leq d$, define 
$$V_l=\bigwedge^l \R^{d}, \qquad   V= \bigoplus_{l=1}^{d} V_l .$$
Define an action of $G$ on $V$ (resp. $V_l$) via the map $g \mapsto \bigoplus_{l=1}^{d} \bigwedge^l g$ (resp. $g \mapsto \wedge^l g$). Let $\{\bfe_1, \ldots , \bfe_{d}\}$ denote the standard basis of $\R^{d}$. For each index set $I = \{ i_1 < \cdots < i_l \} \subset \{1,\dots,d\}$, we define
  \begin{align} \label{eqn: basis elements}
  	{\bfe}_I := {\bfe}_{i_1} \wedge \cdots \wedge {\bfe}_{i_l}.
  \end{align}
  The collection of monomials $\bfe_I$ with $\#I =l$, gives a basis of $V_l = \bigwedge^l \R^{d}$ for each $ 1\leq l\leq d$.
  For $v\in V$ and each index set $I$, we denote by $v_I \in \R$, the unique value so that $v= \sum_{J}v_J\bfe_J$, where the sum is taken over all index sets $J$. We define a {\em norm} $\|.\|$ on $V$ as
  \begin{align}
  \label{eq: def ||||}
      \|v\|= \max_{I} |v_I|,
  \end{align}
  where the maximum is taken over all index sets $I$. For $g \in G$, we define
    \begin{align}
        \label{eq: def operator norm v}
        \|g\|:= \sup \left\{ \|gv\|: v\in V, \|v\| = 1 \right\}.
    \end{align}
    Also, for any compact subset $Q \subset G$, we define
    \begin{align} \label{eq: norm Q}
          \|Q\| = \sup  \{ \|g\|, \|g^{-1}\|: g \in Q \}.
     \end{align}

For $1 \leq l \leq d$, we define $V_l^+$ to be the subspace of $V_l$ spanned by $\bfe_I$, where $I$ varies over the index sets satisfying $\#(I \cap \{1, \ldots, m\}) = \min\{l,m\}$. Similarly, define $V_l^-$ to be the subspace of $V_l$ spanned by $\bfe_I$, where $I$ varies over the index sets satisfying $\#(I \cap \{1, \ldots, m\}) \neq \min\{l,m\}$. Also define $\pi_{l+} $ (resp. $\pi_{l-}$) as the natural projection map from $V_l$ onto $V_l^+$ (resp. $V_l^-$). Note that for all $\theta \in \Mat$, we have $u(\theta)$ acts trivially on $V_l^+$, i.e., $u(\theta)|_{V_{l}^+} = \mathrm{Id}_{V_{l}^+}$. We also define for $1 \leq l \leq d-1$, $w_l$ as the least $w>0$ such that the subspace $V_{l,w}^+= \{v \in V_{l}^+: g_tv = t^wv\}$ is non-empty. It is easy to see that
\begin{align}
    w_l &= \begin{cases}
         a_m + \cdots +a_{m-l+1} \text{ if } l \leq m, \\
         1-(b_1 + \cdots + b_{l-m}) \text{ otherwise. }
    \end{cases} \nonumber\\
\end{align}

\subsection{Lattice Covolumes}
\label{subsec: Covolume of Lattice}
For a discrete subgroup $\Lambda$ of $\R^{m+1}$ of rank $l \geq 1$, we define $v_{\Lambda} \in V_l/\{\pm 1\}$ as $v_1 \wedge \cdots \wedge v_l$, where $v_1, \ldots, v_l$ is a $\Z$-basis of $\Lambda$. Note that the definition of $v_{\Lambda}$ is independent of the choice of basis $v_1, \ldots, v_l$. We define $\|\Lambda\|$ as
        \begin{align}
            \label{eq: def lattice covol}
            \|\Lambda\| = \|v_{\Lambda}\|,
        \end{align}
        where $\|.\|$ on $V_l$ is defined as in \eqref{eq: def ||||}. We also define $\|\{0\}\|=1$.

For $\Lambda \in \X$, we define $P(\Lambda)$ as the set of all \textbf{primitive} subgroups of the lattice $\Lambda$, i.e, the subgroups $L$ of the lattice $\Lambda$ satisfying $L = \Lambda \cap \mrm{span}_\R (L) $, where $\mrm{span}_\R (L)$ is the smallest vector subspace of $\R^d$ containing $L$.

  \begin{lem}[{\cite[Lem.~5.6]{EMM98}}]
    \label{lem EMM}
    There exists a constant $D > 0$ such that the following inequality holds. For all $\Lambda \in \X$ and for all $\Lambda_1, \Lambda_2 \in P(\Lambda)$, we have:
    \begin{align}
        \label{eq: EMM98}
        \|\Lambda_1 \cap \Lambda_2\| \|\Lambda_1 + \Lambda_2\| \leq D \|\Lambda_1\| \|\Lambda_2\|.
    \end{align}
\end{lem}

\begin{rem}
    In \cite{EMM98}, inequality \eqref{eq: EMM98} is established with $D = 1$, but the norm $\|\Lambda\|$ is defined differently. There, $\|\Lambda\|$ is taken as $\|v_{\Lambda}\|$, where $\|.\|$ on $V_l = \wedge^l \R^d$ is the norm induced by the Euclidean norm on $\R^d$. Since any two norms on a finite-dimensional vector space are equivalent, it follows that \eqref{eq: EMM98} holds for some sufficiently large $D$ under our current definition of $\|\Lambda\|$.
\end{rem}


\section{The Critical Exponent}
\label{sec: Critical Exponent}

For each $1 \leq l \leq d-1$, we define the $l$-th critical exponent $\zeta_l(\mu)$ of a measure $\mu$ on $\Mat$ as the supremum of all $\gamma \geq 0$ for which there exists a constant $C_{\gamma,l}' > 0$ such that, for every $v = v_1 \wedge \cdots \wedge v_l \in V_l$ with $\|v\| = 1$ and $r \in \Xi$, the following inequality holds:
\begin{align}
   \int_{\Mat} \frac{1}{\|\pi_{l+}(u(\theta)v)\|^{\gamma}} \, d\mur(\theta) < C_{\gamma,l}'.
\end{align}

\begin{prop}
\label{Critical Exponent is positive}
    For all $1 \leq l \leq d-1$, we have $\zeta_l(\mu) >0$, where $\mu$ is defined as in \eqref{eq: def mu}.
\end{prop}

We will need the following lemmas to prove Proposition \ref{Critical Exponent is positive}.
\begin{lem}
\label{lem: General Critical Expo}
    Suppose $p, q, l \in \N$ are such that $l \leq p + q$. Let $S$ be the set of all subsets of $\{1, \ldots, p + q\}$ of cardinality $l$. Also, define 
    $$
    S^+ = \{I \in S : \#(I \cap \{1, \ldots, p\}) = \min\{p, l\}\}.
    $$ 
    Fix $z_1, \ldots, z_{p+q} \in \R^l$ such that
    $$
    \max\{|z_I| : I \in S\} = 1,
    $$
    where $z_I = z_{i_1} \wedge \cdots \wedge z_{i_l} \in \R$ for $I = \{i_1 < \cdots < i_l\}$. For $\theta \in \M_{p \times q}(\R)$, we define
    \begin{align*}
         z_i^\theta = \begin{cases}
            z_i + \sum_{1 \leq j \leq q} \theta_{ij} z_{p+j} & \text{if } i \leq p, \\
            z_i & \text{otherwise},
        \end{cases} \\
        z_I^\theta = z_{i_1}^\theta \wedge \cdots \wedge z_{i_l}^\theta \text{ for } I = \{i_1 < \cdots < i_l\}.
    \end{align*}
    Suppose $\nu^\rho = \prod_{\substack{1 \leq i \leq p \\ 1 \leq j \leq q}} \nu_{ij}^\rho$ is a family of probability measures on $\M_{p \times q}(\R)$, where $\rho$ varies over a fixed set $F$. Assume there exist constants $\lambda > 0$, $C_1 > 0$, and $C_2 > 0$ such that for all $1 \leq i \leq p$, $1 \leq j \leq q$, and $\rho \in F$, the following hold for all $x \in \R$ and $y>0$:
    \begin{align}
        \label{eq: n1}
        \nu_{ij}^\rho([x - y, x + y]) &\leq C_1 y^\lambda, \\
        \label{eq:n5} 
        \supp(\nu_{ij}^\rho) &\subset [-C_2, C_2].
    \end{align}
    Then there exist constants $C_3 > 0$ and $\delta > 0$, independent of $z_1, \ldots, z_{p+q}$, such that for all $\rho \in F$ and $0 < \e < 1$, we have
    \begin{align}
        \label{eq: n2}
        \nu^\rho\left(\left\{\theta \in \M_{p \times q}(\R) : \max\{|z_I^\theta| : I \in S^+\} \leq \e \right\}\right) \leq C_3 \e^\delta.
    \end{align}
\end{lem}

\begin{proof}
    We will use induction on the value of the difference $p+q-l$ to prove the statement. For the base case, where $p+q=l$, it is straightforward to observe that $S^+ = S$ and $z_I^\theta = z_I$ for all $I$. Thus, \eqref{eq: n2} holds trivially. 

Now, assume that the theorem is valid for all $p, q, l$ satisfying $p+q-l < k$. We proceed to prove the case where $p+q-l = k$. Before moving further, we introduce some necessary notation.

Define  
\begin{align}
    \label{eq: w 6} 
    v_{z^\theta} = ((z^\theta_1)_1\bfe_1 + \cdots + (z^\theta_{p+q})_1 \bfe_{p+q})\wedge \cdots \wedge ((z^\theta_1)_l\bfe_1 + \cdots + (z^\theta_{p+q})_l \bfe_{p+q}) \in \wedge^l\R^{p+q},
\end{align}
where $\bfe_1, \ldots, \bfe_{p+q}$ is the standard basis of $\R^{p+q}$. Observe that the collection of all $\bfe_I := \bfe_{i_1}\wedge \cdots \wedge \bfe_{i_l}$, for index sets $I = \{i_1 < \cdots < i_l\}$, forms a basis of $\wedge^l\R^{p+q}$. For $x \in \wedge^l\R^{p+q}$, we denote by $x_I \in \R$ the unique real numbers such that $x = \sum_{I \in S}x_I \bfe_I$. Then, it is straightforward to see that  
\begin{align}
    \label{eq: n 100}
    (v_{z^\theta})_I = z^\theta_I.
\end{align}

We also define the norm $\|.\|$ on $\wedge^l \R^{p+q}$ as $\|x\|= \max_{I \in S}|x_I|$. Additionally, consider the action of $\SL_{p+q}(\R)$ on $\wedge^l \R^{p+q}$ via the map $g \mapsto \wedge^l g$. For $\theta \in \M_{p \times q}(\R)$, define $u'(\theta)$ as  
\begin{align}
    \label{eq: n 7}
    u'(\theta) = \begin{pmatrix}
        I_p & \theta \\ & I_q
    \end{pmatrix}.
\end{align}
It is then easy to verify that  
\begin{align}
    \label{eq:n 8}
    v_{z^\theta} = u'(\theta)v_z.
\end{align}

Also, define $C_4 > 0$ as follows:
\begin{align}
    \label{eq:n 9}
    C_4 = \sup\{\|u'(\theta)x\| :\, \theta \in \M_{p \times q}([-C_2, C_2]),\;  x\in \wedge^l \R^{p+q}, \; \|x\|\leq 1\} < \infty.
\end{align}

    We divide the statement into two cases.

\noindent {\bf Case 1:} For all $I \notin S^+$, we have $|z_I| \leq 1/2C_4$. Define $\sigma^+: \wedge^l\R^{p+q} \rightarrow \sspan\{\bfe_I: I \in S^+\}$ and $\sigma^-: \wedge^l\R^{p+q} \rightarrow \sspan \{\bfe_I: I \notin S^+\}$ as the natural projection maps. Clearly, we have $\|\sigma^+(v_z)\|=1$ and $\|\sigma^-(v_z)\| \leq 1/2C_4$. 

In this case, for all $\theta \in \M_{p \times q}([-C_2, C_2])$, the following holds:
\begin{align*}
    \max_{I \in S^+}|z_I^\theta| &= \max_{I \in S^+}|(v_{z^\theta})_I| \quad \text{using \eqref{eq: n 100}} \\
    &= \|\sigma^+(v_{z^\theta})\| \\
    &= \|\sigma^+(u'(\theta)v_{z})\| \quad \text{using \eqref{eq:n 8}} \\
    &= \|\sigma^+(u'(\theta) (\sigma^+(v_{z}) + \sigma^-(v_{z})) )\| \\
    &\geq \|\sigma^+(u'(\theta) (\sigma^+(v_{z}))) \| - \| \sigma^+(u'(\theta) (\sigma^-(v_{z}))) \| \\
    &\geq \|\sigma^+(v_z)\| - \| u'(\theta) (\sigma^-(v_{z})) \| \quad \text{as $u'(\theta)$ acts trivially on $\sspan \{\bfe_I: I \in S^+\}$} \\
    &\geq 1 - C_4 \|\sigma^-(v_{z})\| \\
    &\geq 1 - \frac{1}{2} = \frac{1}{2}.
\end{align*}
Thus, \eqref{eq: n2} holds for all $\gamma > 0$ with a corresponding $C_3 \geq 2^\gamma$.

    \noindent {\bf Case 2:} In this case, there exists an $I \notin S^+$ such that $|z_I| > 1/2C_4$.  Assume that the maximum $\max_{I \notin S^+}|z_I|> 1/2C_4$ is achieved for the index set $J$. Since $J \notin S^+$, we have $\{1, \ldots, p\} \setminus J \neq \emptyset$. For simplicity, assume that $1 \notin J$. Let
    $$
    c= \max\{|z_I|: 1 \notin I, I \in S \}> \frac{1}{2} C_4.
    $$
    Now, consider the $p+q-1$ vectors $c^{-1/l}z_2, \ldots, c^{-1/l}z_{p+q} \in \R^l$. These vectors satisfy the following condition:  
\[
\max\{ |c^{-1/l}z_{i_1} \wedge \cdots \wedge c^{-1/l}z_{i_l}| : 1 \notin I, I \in S \} = 1.
\]
For all $\rho \in F$, define the measure $\nu^{\rho'} := \prod_{\substack{2 \leq i \leq p \\ 1 \leq j \leq q}} \nu^\rho_{ij}$ on $\M_{(p-1) \times q}(\R)$. Let 
$$S^{'+} = \{I: \#(I \cap \{2, \ldots, p\}) = \min\{p-1, l\}\}.$$ Since $(p-1) + q - l = k-1$, the induction hypothesis applies, giving constants $\delta_1 > 0$ and $C_5$ such that for all $\rho \in F$ and $\epsilon > 0$, we have
\begin{align*}
    \nu^{\rho'}(\{\theta \in \M_{(p-1) \times q}(\R):  \max\{c^{-1}|\wedge_{j=1}^l (z_{i_j} + \sum_{m=1}^q \theta_{i_j,m} z_m)|: I= \{i_1< \cdots < i_l\} \in S^{'+}\} \leq \epsilon \}) \leq C_5\epsilon^{2\delta_1},
\end{align*}
which implies
\begin{align}
\label{eq: n 10}
    \nu^{\rho}(\{\theta \in \M_{p\times q}(\R):  \max\{|z_I^\theta|: I \in S^{'+}\} \leq \epsilon^{1/2} \}) \leq C_5c^{-2{\delta_1}}\epsilon^{\delta_1}.
\end{align}
Now, if $l \leq p-1$, then $S^{'+} \subset S^+$, and \eqref{eq: n 10} implies \eqref{eq: n2} for all $\delta \leq \delta_1$. Hence, we may assume that $l \geq p$. In this case, the set $R= \{\theta \in \M_{p \times q}(\R):  \max\{|z_I^\theta|: I \in S^+\} \leq \epsilon \}$ is contained in the union $R_0 \cup (\cup_{I \in S^{'+}} R_I)$, where
\begin{itemize}
    \item $R_0 = \left\{\theta \in \mathrm{M}_{p \times q} : 
     \max\{|z_I^\theta|: I \in S^{'+}\} \leq \epsilon^{1/2}
    \right\}$,
    \item For each $I \in S^{'+}$, $R_I$ is the set of all $\theta \in R \setminus R_0$ such that  $|z_I^\theta| \geq \epsilon^{1/2}$.
\end{itemize}
Next, let us estimate the measure of $R_I$ for a fixed $I$. Fix $1 \leq i \leq q$ such that $p+i \in I$. For $\theta \in R_I$, the condition $\max\{|z_J^\theta|: J \in S^+\} \leq \epsilon$ implies $|z_{I \cup \{1\} \setminus \{p+i\} }^\theta| \leq \epsilon$. Observe that
\begin{align}
    z_{I \cup \{1\} \setminus \{p+i\} }^\theta &= (z_{1} + \sum_{j=1}^q \theta_{1j}z_{p+j}) \wedge z_{I \setminus \{p+i\}}^\theta \nonumber \\
    &= z_{1} \wedge z_{I \setminus \{p+i\}}^\theta + \sum_{j=1}^q \theta_{1j}z_{p+j} \wedge z_{I \setminus \{p+i\}}^\theta \nonumber \\
    &= z_{1} \wedge z_{I \setminus \{p+i\}}^\theta + \sum_{\substack{j \in \{1, \ldots, q\} \\ p+j \notin I \setminus \{p+i\}}} \theta_{1j} \alpha_{j} z_{(I  \setminus \{p+i\}) \cup \{p+j\}}^\theta, \label{eq: n 11}
\end{align}
where $\alpha_j \in \{\pm 1\}$ such that $z_{p+j} \wedge z_{I \setminus \{p+i\}}^\theta= \alpha_{j} z_{(I  \setminus \{p+i\}) \cup \{p+j\}}^\theta$ for all $j \in \{1, \ldots, q\}$ satisfying $p+j \notin I \setminus \{p+i\} $.

Notice that $z_{(I  \setminus \{p+i\}) \cup \{p+j\}}^\theta$ is independent of $\theta_{1l}$ for $1 \leq l \leq q$. Thus, from \eqref{eq: n 11}, for fixed $\theta_{l_1l_2}$ with $(l_1,l_2) \neq (1,i)$, the set of all $\theta_{1i}$ such that $\theta$ belongs to $R_I$ lies in an interval of radius at most $\epsilon/|z_I^\theta| \leq \epsilon^{1/2}$. By \eqref{eq: n1}, this interval has $\nu^\rho_{1i}$ measure less than $C_1\epsilon^{\lambda/2}$. Using Fubini's theorem, we get $\nu^\rho(R_I) \leq C_1 \epsilon^{\lambda/2}$.

Thus, we have $\nu^\rho(R) \leq \nu^\rho(R_0) + \sum_{I \in S^{'+}} \nu^\rho(R_I) \ll \e^{\delta_1} + \e^{\lambda/2}$, where implied constant is independent of $z$. Thus, in this case, $\gamma = \min\{\lambda/2, \delta_1\}$ works. By induction, the lemma holds.

\end{proof}

\begin{lem}
\label{lem: Integral estimate}
For every $1 \leq l \leq d-1$, there exist constants $\gamma > 0$ and $L_{\gamma} > 0$ such that the following holds. For any $v = v_1 \wedge \cdots \wedge v_l \in V_l$ with $\|v\| = 1$ and any $r \in \Xi$, we have:
\begin{align}
\label{eq: u 0}
    \mur\left(\left\{ \theta \in \Mat : \|\pi_{l+}(u(\theta)v)\| \leq \epsilon \right\}\right) \leq L_{\gamma} \epsilon^{\gamma}.
\end{align}
\end{lem}
\begin{proof}
     Note that using Proposition \ref{prop: important IFS} and the fact that $\mu$ is a probability measure, we may assume that there exist constants $C_1, C_2, \lambda > 0$ such that for all $1 \leq i \leq m$, $1 \leq j \leq n$, $r \in \Xi$, $x \in \R$ and $y>0$, we have
\begin{align*}
    \supp(\mur_{ij}) &\subset [-C_1, C_1], \\
    \mur_{ij}([x-y, x+y]) &\leq C_2 y^\lambda.
\end{align*}

Now fix $1 \leq l \leq d$. For $\bfv = (v_1, \ldots, v_l) \in (\R^d)^l$, define $z_{1,\bfv}, \ldots, z_{d,\bfv}$ as vectors in $\R^l$ such that the $j$-th entry of $z_{i,\bfv}$ equals $(v_j)_i$ (the $i$-th entry of $v_j$). For $\theta \in \M_{m \times n}(\R)$, we let $\bfv^\theta = (u(\theta)v_1, \ldots, u(\theta)v_l)$. Let $v = v_1 \wedge \cdots \wedge v_l \in \wedge^l \R^d$. Then it is straightforward to see that
\begin{align}
    z_{i,\bfv^\theta} &= 
    \begin{cases}
        z_{i,\bfv} + \sum_{j=1}^n \theta_{ij} z_{m+j,\bfv}, & \text{if } i \leq m, \\
        z_{i,\bfv}, & \text{otherwise},
    \end{cases} \label{eq: m 1} \\
    z_{i_1, \bfv^\theta} \wedge \cdots \wedge z_{i_l, \bfv^\theta} &= (u(\theta)v)_{\{i_1, \ldots, i_l\}}, \label{eq: m 2}
\end{align}
for all $i_1 < \cdots < i_l$, where in the last expression we identify $\wedge^l \R^l$ with $\R$. Hence, using \eqref{eq: m 2}, we obtain
\begin{align*}
    &\{ \theta \in \Mat : \|\pi_{l+}(u(\theta)v)\| \leq \epsilon \} \\
    &= \{ \theta \in \Mat : \max\{|\wedge_{i \in I} z_{i,\bfv^\theta}| : \#(I \cap \{1, \ldots, m\}) = \min\{m, l\} \} \leq \epsilon \}.
\end{align*}
Using \eqref{eq: m 1}, it follows that the lemma can be deduced from Lemma \ref{lem: General Critical Expo} with $p = m$, $q = n$, $l = l$, and the measures $\mur$ replacing $\nu^\rho$.
\end{proof}

The following lemma will be needed in the proof of the next lemma and is a simple application of Fubini's Theorem.
\begin{lem}
    \label{Fubini Theorem}
     Let $\nu$ be a Borel measure and $f$ a non-negative Borel function on a separable metric space $X$. Then,
     $$
     \int_X f\;d\nu = \int_0^\infty \nu\left(\{x\in X: f(x)\geq R \}\right) \;dR
     $$
\end{lem}

\begin{lem}
    \label{lem: Finding the Critical Exponent}
    Fix $1 \leq l \leq d-1$. Assume that there exist $\gamma > 0$ and $L_{\gamma} > 0$ such that for all $v = v_1 \wedge \cdots \wedge v_l \in V_l$ with $\|v\| = 1$ and $r \in \Xi$, the following holds:
    $$
    \mur\left(\{ \theta \in \Mat : \|\pi_{l+}(u(\theta)v)\| \leq \epsilon \} \right) \leq L_{\gamma} \epsilon^{\gamma}.
    $$
    Then $\zeta_l(\mu) \geq \gamma$.
\end{lem}

\begin{proof}
    Note that for all $0 < \rho < 1$, $r \in \Xi$, and $v = v_1 \wedge \cdots \wedge v_l \in V_l$, the following holds:
     \begin{align*}
        &\int_{\Mat} \frac{1}{\|\pi_{l+}(u(\theta)v)\|^{\gamma \rho}} \, d\mur(\theta) \\
        &= \int_{0}^\infty \mur\left( \theta: \frac{1}{\|\pi_{l+}(u(\theta)v)\|^{\gamma \rho}} \geq R  \right) \, dR  (\text{ using Lemma \ref{Fubini Theorem}})\\
        &= \int_{0}^1 \mur\left( \theta: \frac{1}{\|\pi_{l+}(u(\theta)v)\|^{\gamma \rho}} \geq R  \right) \, dR + \int_{1}^\infty \mur\left(  \theta : \frac{1}{\|\pi_{l+}(u(\theta)v)\|^{\gamma \rho}} \geq R  \right) \, dR  \\
        &\leq \int_0^1 1 dR + \int_{1}^\infty \mur\left(  \theta: \|\pi_{l+}(u(\theta)v)\| \leq R^{-1/(\gamma \rho)}  \right) \, dR  \\
        &\leq 1 + \int_{1}^\infty L_{\gamma} R^{-1/ \rho} \, dR  < \infty \quad \text{using fact that $\rho<1$}.
    \end{align*}
    By the definition of $\zeta_l(\mu)$, we conclude that $\zeta_l(\mu) \geq \rho \gamma$. Since this holds for all $0 < \rho < 1$, the lemma follows.
\end{proof}

\begin{proof}[Proof of Proposition \ref{Critical Exponent is positive}]
    The proposition follows directly from Lemmas \ref{lem: Integral estimate} and \ref{lem: Finding the Critical Exponent}.
\end{proof}

\section{Critical Exponent: Special Cases}
\label{sec: Critical Exponent: Special Cases}
In this subsection, we will provide non-trivial lower bounds for the critical exponents of $\mu$ in the following cases.

\subsection{The Matrix Case}
In this case, we assume that $\Kcal = \M_{m \times n}([0, 1])$ and work with the dyadic model. 
That is, for all $1 \leq i \leq m$ and $1 \leq j \leq n$, we assume that $\Phi_{ij}$ contains the following two contracting similarities: $x \mapsto x/2$ and $x \mapsto (x+1)/2$. The main goal of this subsection is to prove the following lemma.

\begin{lem}
    \label{lem: Critical Exponent Real Case}
   If $\Kcal= \M_{m \times n}([0,1])$, then 
   $$
   \zeta_l(\mu) \geq  \begin{cases}
       \frac{m}{l} \quad \text{if } l \leq m, \\
       \frac{n}{m+n-l} \quad \text{if } m < l \leq d-1.
   \end{cases}
   $$
\end{lem}

To prove this lemma, we will use some results from \cite{KKLM}. First, we define the following for $t\in \R$:

\begin{align}
    \label{eq: def h t}
    h_t = \begin{pmatrix}
        e^{nt}I_m & 0 \\ 
        0 & e^{-mt}I_n
    \end{pmatrix}.
\end{align}

Now, we state the required results from \cite{KKLM}:

\begin{prop}[{\cite[Prop.~3.1]{KKLM}}]
    \label{KKLM Prop 3.1}
    For any $l \in \{1, \ldots, d-1\}$, define $\beta_l$ as follows:
    $$
    \beta_l = 
    \begin{cases} 
        \frac{m}{l} & \text{if } l \leq m, \\
        \frac{n}{m+n-l} & \text{if } l > m.
    \end{cases}
    $$
    Then, there exists a constant $c > 0$ (depending only on $m$ and $n$) such that for any $t \geq 1$ and any $v = v_1 \wedge \cdots \wedge v_l \in \wedge^l \R^d$, we have
    $$
    \int_{\SO(d)} \|h_t k v\|^{-\beta_l} \, dk \leq c t e^{-mnt} \|v\|^{-\beta_l}.
    $$
\end{prop}

\begin{lem}[{\cite[Lem.~3.5]{KKLM}}]
    \label{KKLM Lem 3.5}
    There exists a neighborhood $V$ of the identity in $\Mat$ such that for any $s_0 \in \Mat$, $t, \beta > 0$, $l \in \{1, \ldots, d-1\}$, and a vector $v = v_1 \wedge \cdots \wedge v_l \in \wedge^l \R^d$, we have 
    $$
    \int_{V + s_0} \|h_t u(\theta) v\|^{-\beta} \, d\theta \ll (1 + \|s_0\|)^\beta \int_K \|h_t k v\|^{-\beta},
    $$
    where the implied constant depends only on $m$, $n$, and $\beta$.
\end{lem}

Using the above results, we can now prove the lemma \ref{lem: Critical Exponent Real Case}.

\begin{proof}[Proof of Lemma \ref{lem: Critical Exponent Real Case}]
For $1 \leq l \leq d-1$, define
$$
    \beta_l = 
    \begin{cases} 
        \frac{m}{l} & \text{if } l \leq m, \\
        \frac{n}{m+n-l} & \text{if } l > m.
    \end{cases}
$$
Now fix $1 \leq l \leq d-1$ and  $v = v_1 \wedge \cdots \wedge v_l \in \wedge^l \R^d$ with $\|v\| = 1$. Also, fix $0 < \e < 1$ and let $t > 1$ be such that $e^{t(m+n)} \e = K$, where $K = \|\{u(\theta): \theta \in \M_{m \times n}([0, 2]) \}\|$. Let us define the set $E(v, \e) = \{\theta \in \Mat : \|\pi_{l+}(u(\theta)v)\| < \e\}$.

Now, let $V$ be a neighbourhood of the identity in $\Mat$ such that Lemma \ref{KKLM Lem 3.5} holds. Let $s_0, \ldots, s_p \in \Mat([0, 2])$ be a finite set such that $\Mat([0, 2]) \subset \cup_{i=1}^p (V + s_i)$. From the definition of $\mur$, it is easy to see that $\supp(\mur) \subset \M_{m \times n}([0, 2])$. Thus, we have
\begin{align}
    \int_{\Mat} \|h_t u(\theta)v\|^{-\beta_l} \, d\mur(\theta) &\ll \sum_{i=0}^p \int_{V + s_i} \|h_t u(\theta)v\|^{-\beta_l} \, d\theta  \nonumber \\
    &\ll \sum_{i=0}^p (1 + \|s_i\|)^{\beta_l} \int_K \|h_t k v\|^{-\beta_l} \, dk \quad \text{(using Lemma \ref{KKLM Lem 3.5})} \nonumber \\
    &\ll \int_K \|h_t k v\|^{-\beta_l} \, dk \nonumber \\
    &\ll t e^{-mnt} \|v\|^{-\beta_l} = t e^{-mnt}, \label{eq: o 1}
\end{align}
where the implied constant depends only on $m, n$.

Next, for all $\theta \in E(v, \e) \cap \M_{m \times n}([0, 2])$, we have the following estimates:
\begin{align*}
    \|h_t u(\theta) v\| &= \max\left\{ |(h_t u(\theta) v)_I| : \#I = l \right\} \\
    &= \max\left\{ e^{(\#(I \cap \{1, \ldots, m\}) nt) - (\#(I \cap \{m+1, \ldots, d\}) mt)} |(u(\theta) v)_I| : \#I = l \right\} \\
    &\leq \max \Bigg\{
        e^{\frac{mnt}{\beta_l}} \|\pi_{l+}(u(\theta) v)\|, \;
        e^{\frac{mnt}{\beta_l} - (m+n)t} |(u(\theta) v)_I| 
        : \#(I \cap \{1, \ldots, m\}) < \min\{l, m\}
      \Bigg\} \\
    &\leq \max \Bigg\{
        e^{\frac{mnt}{\beta_l}} \e, \;
        e^{\frac{mnt}{\beta_l}} \frac{\e}{K}
        \|u(\theta)\| \|v\|
      \Bigg\} \\
    &\leq e^{\frac{mnt}{\beta_l}} \e.
\end{align*}

Thus, we conclude that
$$
E(v, \e) \cap \M_{m \times n}([0, 2]) \subset \{\theta \in \Mat : \|h_t u(\theta) v\| < e^{\frac{mnt}{\beta_l}} \e\} \quad \text{for all } r \in \Xi.
$$
Using \eqref{eq: o 1} and applying Markov's inequality, we obtain the following bound for all $\rho < \beta_l$:
\begin{align*}
    \mur(E(v, \e)) &\leq e^{mnt} \e^{\beta_l} \int_{\Mat} \|h_t u(\theta) v\|^{-\beta_l} \, d\mur(\theta) \\
    &\ll e^{mnt} \e^{\beta_l} t e^{-mnt} = t \e^{\beta_l}  \\
    &=  \frac{(\log K - \log \e)}{m+n} \e^{\beta_l}  \\
    &\ll \e^{\rho},
\end{align*}
where the implied constant depends only on $m, n, \rho$. Thus, by Lemma \ref{lem: Finding the Critical Exponent}, we conclude that $\zeta_l(\mu) \geq \rho$. Since this holds for all $\rho < \beta_l$, the lemma follows.
\end{proof}

\subsection{Case $n=1$}
\begin{lem}
\label{lem: n 1 Critical Expo}
    For $n=1$, we have $\zeta_l(\mu)\geq \min\{\sum_{i \in I}s_{i1}: \#I= d-l\}$ for all $1 \leq l \leq d-1$.
\end{lem}
\begin{proof}
Fix \(1 \leq l \leq d-1\) and define \( \gamma = \min\left\{\sum_{i \in I}s_{i1}: \#I = d - l  \right\} \). Fix \( v = v_1 \wedge \cdots \wedge v_l \in \wedge^l \mathbb{R}^d \) such that \( \|v\| = 1 \). For all \( 0 < \epsilon < 1 \), define the set
$$
E(v, \epsilon) = \left\{ x \in \mathbb{R}^m : \|\pi_+(u(x)v)\| < \epsilon \right\}.
$$
Let
$$
K = \left\| \left\{ u(x): x \in \bigcup_{r \in \Xi} \supp(\mur) \right\} \right\|.
$$
Since \( \mu \) has compact support, it follows that \( K < \infty \). Additionally, let \( \lambda > 0 \) be sufficiently large so that for all \( r \in \Xi \) and \( 1 \leq i \leq m \), the following holds for all \( x, y \in \mathbb{R} \):
\begin{align}
    \label{eq: c 1}
    \mur_{i1}([x - y, x + y]) \leq \lambda y^{s_{i1}},
\end{align}
where the existence of such a \( \lambda \) follows from Proposition \ref{prop: important IFS} and the definition of \( \mur_{ij} \).

To estimate \( \mur(E(v, \epsilon)) \), we consider two cases:

\noindent {\bf Case 1:} \( \| \pi_{l-}(v) \| \leq \frac{1}{2K} \). In this case, we have \( \|\pi_{l+}(v)\| = \|v\| = 1 \). Thus, for \( x \in \bigcup_{r \in \Xi} \supp(\mur) \), we get
$$
\|\pi_{l+}(u(x)v)\| \geq \|\pi_{l+}(u(x) \pi_{l+}(v))\| - \|\pi_{l+}(u(x) \pi_{l-}(v))\| \geq \|\pi_{l+}(v)\| - \|u(x) \pi_{l-}(v)\| \geq 1 - K \cdot \frac{1}{2K} = \frac{1}{2}.
$$
Therefore, in this case, we have
$$
\mur(E(v, \epsilon)) \leq 2^\gamma \epsilon^\gamma.
$$

\noindent {\bf Case 2:} \( \| \pi_{l-}(v) \| > \frac{1}{2K} \). In this case, there exists a subset \( \{m+1\} \subset I \subset \{1, \ldots, m+1\} \) of cardinality \( l \) such that \( |v_I| \geq \frac{1}{2K} \). Let us define, for each \( i \notin I \), the set
$$
I_i = I \cup \{i\} \setminus \{m+1\}.
$$

Define the set
$$
E'(v, \epsilon) = \left\{ x \in \mathbb{R}^m : \text{for all } i \notin I, \ |(\pi_+(u(x)v))_{I_i}| < \epsilon \right\}.
$$
It is clear that \( E(v, \epsilon) \subset E'(v, \epsilon) \).

Now, to estimate \( \mur(E'(v, \epsilon)) \), we note that by explicit computation, for every \( J \subset \{1, \ldots, m\} \), the following holds:
$$
(\pi_+(u(x)v))_J = v_J + \sum_{j \in J} \bfe_{j,J} v_{J \cup \{m+1\} \setminus \{j\}} x_{j1},
$$
for some \( \bfe_{j,J} \in \{\pm 1\} \). Thus, for all \( i \notin I \), we have
$$
(\pi_+(u(x)v))_{I_i} = \bfe_{i,I_i} v_I x_{i1} + \left( v_{I_i} + \sum_{j \in I \setminus \{m+1\}} \bfe_{j,I_i} v_{I_i \cup \{m+1\} \setminus \{j\}} x_{j1} \right).
$$
If we fix \( x_{j1} \) for \( j \in I \setminus \{m+1\} \), then for all \( i \notin I \) and \( \epsilon > 0 \), the condition \( |(\pi_+(u(x)v))_{I_i}| < \epsilon \) is equivalent to the condition that \( x_{i1} \) belongs to an interval of size at most \( \frac{\e}{|v_I|} < 2K\epsilon \), which has \( \mur_{i1} \)-measure less than \( \lambda (2K\epsilon)^{s_{i1}} \) (using equation \eqref{eq: c 1}). By Fubini's theorem, we then obtain
$$
\mur(E(v, \epsilon)) \leq \mur(E'(v, \e)) \leq \lambda^{m+1-l} (2K)^{\sum_{i \notin I} s_{i1}} \e^{\sum_{i \notin I} s_{i1}} \leq \lambda^{m+1-l} (2K)^{\sum_{i \notin I} s_{i1}} \e^\gamma.
$$
Thus, by Lemma \ref{lem: General Critical Expo}, we conclude that \( \zeta_l(\mu) \geq \gamma \). This completes the proof of the lemma.

\end{proof}

\subsection{Case $m=1$}
\begin{lem}
\label{lem: m 1 Critical Expo}
    For $m=1$, we have $\zeta_l(\mu)\geq \min\{\sum_{i \in I}s_{1i}: \#I= l\}$ for all $1 \leq l \leq d-1$.
\end{lem}
\begin{proof}
  Fix $1 \leq l \leq d-1$. For $\theta \in \M_{1 \times n}(\R)$, we set 
  $$
  u_2(\theta) = \begin{pmatrix}
      I_n & \begin{pmatrix}
          \theta_{11} \\ \vdots \\ \theta_{1n}
      \end{pmatrix} \\ & 1
  \end{pmatrix}.
  $$
  Let us define the map $*: \wedge^l\R^d \rightarrow \wedge^{d-l}\R^d$ as the unique linear map so that $*(\bfe_I) = \bfe_{\{1,\ldots, d\} \setminus I}$ for each index set $I$. Also define $\phi: \R^d \rightarrow \R^d$ as unique linear map so that
  $$
  \phi(\bfe_i) = \begin{cases}
      \bfe_{i-1} \text{ if } i \neq 1, \\
      \bfe_{d} \text{ if } i=1.
  \end{cases}
  $$
  The map $\phi$ gives us a linear map $\wedge^{d-l} \phi: \wedge^{d-l}\R^d \rightarrow \wedge^{d-l}\R^d $. Let us define $\psi = \wedge^{d-l}\R^d \circ *: \wedge^l \R^d \rightarrow \wedge^{d-l}\R^d$. Also, define $f: \M_{1 \times n}(\R) \rightarrow \M_{1 \times n}(\R)$ as $f(\theta_{11}, \ldots, \theta_{1n})= (\theta_{11}, - \theta_{12}, \ldots, (-1)^{n-1}\theta_{1n})$. Then it is easy to see that for all $v \in \wedge^l\R^d$ and $\theta \in \Mat$, we have $\psi(u(\theta)v)= u_2(f(\theta)) \psi(v)$ and $\|\pi_+(u(\theta) v) \|= \max\{|(u_2(f(\theta)) \psi(v))_I|:  \#(I \cap \{1, \ldots,n\}) = \min\{d-l,n\} \}$. Thus, the lemma follows from lemma \ref{lem: n 1 Critical Expo}.
\end{proof}


\section{Height Functions}
\label{sec: Height Functions}

For the remainder of the paper, we fix a sequence $\eta_1, \ldots, \eta_{d-1} \in \mathbb{R}$ such that the following holds:
\begin{align*}
    0 < \eta_i &< \zeta_i(\mu), \quad \text{for } 1 \leq i \leq d-1, \\
    \frac{1}{\eta_{i-j}} + \frac{1}{\eta_{i+j}} &< \frac{2}{\eta_i}, \quad \text{for all } 1 \leq i \leq d-1 \text{ and } 1 \leq j \leq \min\{i, d-i\},
\end{align*}
where we define $\frac{1}{\eta_0} = \frac{1}{\eta_d} := 0$. Additionally, we define the following:
\begin{align*}
    \eta &= \min_{1 \leq l \leq d} w_l \eta_l, \\
    \bfn &= (\eta, \eta_1, \ldots, \eta_{d-1}), \\
    C_{\bfn} &= \max_{1 \leq l \leq d-1} C_{\eta_l, l}'.
\end{align*}

\begin{prop}
    \label{prop: Critical Exponent representation}
    For all $1 \leq l \leq d-1$, $r \in \Xi$, $t>1$ and $v= v_1 \wedge \cdots \wedge v_l \in V_l \setminus \{0\}$, the following holds
    \begin{align}
        \label{eq: Critical Exponent representation}
         \int_{\Mat}  \|g_{t} u(x)v\|^{- \eta_l} \, d\mur(x) \leq C_{\bfn} t^{-\eta } \|v\|^{-\eta_l }.
    \end{align}
\end{prop}
\begin{proof}
    Fix $1 \leq l \leq d-1$, $r \in \Xi$, $t>1$ and $v= v_1 \wedge \cdots \wedge v_l \in V_l \setminus \{0\}$. Without loss of generality, we may assume that $\|v\| =1$. Then we have
    \begin{align*}
        \int_{\Mat}  \|g_{t} u(x)v\|^{-\eta_l} \, d\mur(x) & \leq \int_{\Mat}  \|\pi_+(g_{t} u(x)v)\|^{-\eta_l} \, d\mur(x)  \\
        &\leq \int_{\Mat} t^{-\eta_l w_l } \|\pi_+( u(x)v)\|^{-\eta_l} \, d\mur(x),
    \end{align*}
    where in the last inequality, we used the fact that $t>1$ and that
    $$g_t \bfe_I= t^{ \sum_{i \in I \cap \{1, \ldots, m\}}a_i - \sum_{j \in I \cap \{m+1, \ldots, d\}} b_{j-m} } \bfe_I  \geq t^{w_l} \bfe_I,$$
    for all $I$, such that $\bfe_I \in V_{l}^+ = \pi_{l+}(V_l)$. The proposition now follows from the definition of $C_{\bfn}$ and $\eta$.
\end{proof}

 For every $0 \leq l \leq d$, we define $\varphi_{l}:\X \rightarrow \R$ as
    \begin{align} \label{eq: def varphi l}
    \varphi_{l}(\Lambda) =  \max \{\|\Lambda_l\|^{-1}: \Lambda_l \in P(\Lambda), \operatorname{rank}(\Lambda_l)=l  \},
    \end{align}
    where $\|.\|$ is defined as in \eqref{eq: def lattice covol}. It is easy to see that $\varphi_0 \equiv \varphi_{d} \equiv 1$.

    \begin{prop}
        \label{prop: contraction varphi}
         For all $t >1$, there exists $\xi(t) \geq 1$, such that the following holds for all $1 \leq l \leq d-1$, $r \in \Xi$ and $\Lambda \in {\X}$ 
        \begin{align*}
            \int_{\Mat} {\varphi}_{l}^{ \eta_l  }(g_{t} u(x) \Lambda) \, d\mur({x})
            &\leq C_{\bfn} t^{- \eta } {\varphi}_l^{ \eta_l }(\Lambda) + \xi(t) \left( \max_{1 \leq j \leq \min\{l, d-l\}}   \varphi_{l-j}(\Lambda) \varphi_{l+j}(\Lambda) \right)^{ \eta_l/2 }.
        \end{align*}
    \end{prop}
    \begin{proof}
        Fix $1 \leq l \leq d-1$, $ r \in \Xi$, $t >1$ and $\Lambda \in \X$. Let us define
        $$
        \xi'(t) =  \|\{g_tu(x): x \in \bigcup_{r \in \Xi} \supp(\mur) \}\|.
        $$
        Let $\Lambda_l \in P(\Lambda)$ be a sublattice of rank $l$ such that $\varphi_l(\Lambda)=\|\Lambda_l\|^{-1}$.  Let $v_{\Lambda_l}$ be defined as in Section \ref{subsec: Covolume of Lattice}. We
        claim that for all $x \in \supp(\mur)$, we have
        \begin{align}
            \label{eq: w 8}
            \varphi_l(g_tu(x)\Lambda) \leq \max\left\{ \frac{1}{\|g_tu(x)\Lambda_l\|} , D \xi'(t) \left(  \max_{1 \leq j \leq \min\{l, d-l\}} \varphi_{l-j}(\Lambda) \varphi_{l+j}(\Lambda) \right)^{ 1/2 } \right\},
        \end{align}
        where $D$ is defined as in Lemma \ref{lem EMM}.
        To prove the claim, note that if $\varphi_l(g_tu(x)\Lambda) \neq \|g_tu(x)\Lambda_l\|^{-1}$, then the following holds. In this case, there exists $\Lambda_{l,x,t} \in P(\Lambda)$ with $\Lambda_{l,x,t} \neq \Lambda_l$ such that 
        \begin{align*}
            \varphi_l(g_tu(x)\Lambda) &= \|g_tu(x){\Lambda_{l,x,t}}\|^{-1} \\
            &= \|g_tu(x)v_{\Lambda_{l,x,t}}\|^{-1} \quad \text{ where $v_{\Lambda_{l,x,t}}$ is defined as in Section \ref{subsec: Covolume of Lattice}} \\
            &\leq \frac{\xi'(t)}{\|v_{\Lambda_{l,x,t}}\|} \\
            &\leq  \frac{\xi'(t)}{ \left( \|\Lambda_{l,x,t}\| \|\Lambda_l\| \right)^{1/2} } \\
            &\leq \frac{\xi'(t) D}{\left( \|\Lambda_l \cap \Lambda_{l,x,t}\| \|\Lambda_l+ \Lambda_{l,x,t}\| \right)^{1/2}} \quad (\text{using Lemma \ref{lem EMM}})\\
            &\leq  D.\xi'(t) \left( \max_{1 \leq j \leq \min\{l, d-l\}} \varphi_{l-j}(\Lambda) \varphi_{l+j}(\Lambda) \right)^{1/2}
        \end{align*}
        Thus the claim holds. Let $\xi(t)= \max\{(D \xi'(t))^{\eta_i}: 1 \leq i \leq d-1\}$. Using \eqref{eq: w 8}, we have
        \begin{align*}
            &\int_{\Mat} {\varphi}_{l}^{ \eta_l }(g_{t} u(x) \Lambda) \, d\mur({x}) \\
            &\leq \int_{\Mat} \frac{1}{\|g_tu(x) v_{\Lambda_l}\|^{\eta_l }} \, d\mur({x}) + \xi(t) \left(  \max_{1 \leq j \leq \min\{l, d-l\}} \varphi_{l-j}(\Lambda) \varphi_{l+j}(\Lambda) \right)^{ \eta_l/2 }\\
            &\leq C_{\bfn} t^{- \eta} \frac{1}{\|v_{\Lambda_l}\|^{\eta_l }} + \xi(t) \left( \max_{1 \leq j \leq \min\{l, d-l\}} \varphi_{l-j}(\Lambda) \varphi_{l+j}(\Lambda) \right)^{ \eta_l/2 } \\
            &= C_{\bfn} t^{- \eta } {\varphi}_l^{ \eta_l }(\Lambda) + \xi(t) \left( \max_{1 \leq j \leq \min\{l, d-l\}}  \varphi_{l-j}(\Lambda) \varphi_{l+j}(\Lambda) \right)^{ \eta_l/2 },
        \end{align*}
        where the penultimate inequality follows from Proposition \ref{prop: Critical Exponent representation}. Hence, the proposition follows.
    \end{proof}

Let us define
$$
\an = \min\left\{ 1- \frac{\eta_i}{2}\left(  \frac{1}{\eta_{i-j}}+ \frac{1}{\eta_{i+j}} \right): 1 \leq i \leq d-1, 1 \leq j \leq \min \{i, d-i\} \right\},
$$
where we set $1/\eta_0 = 1/\eta_d = 0$. For $0<\e<1$, we define the function $f_{\e, \bfn}: \X \rightarrow \R$ as 
\begin{align}
    \label{eq:def height fun}
    f_{\e, \bfn}(\Lambda)= \e^{-1} + \sum_{l=1}^{d-1}  \varphi_l^{\eta_l }(\Lambda) .
\end{align}

\begin{prop}
    \label{prop: existence of height function}
    For all $t > 1$, there exists $b= b(t, \bfn) \geq 0$ and $0< \e = \e(t,\bfn )< 1$ such that the following holds for all $\Lambda \in \X$ and $ r \in \Xi$ 
    \begin{align}
        \label{eq: height fn contraction}
        \int_{\Mat} f_{\e,\bfn}(g_tu(x)\Lambda) \, d\mur(x) \leq 2C_{\bfn} t^{-\eta} f_{\e, \bfn}(\Lambda) + b.
    \end{align}
\end{prop}
\begin{proof}
    Fix $t>1$. Let $\xi(t)$ be the constant provided by Proposition \ref{prop: contraction varphi}. Let $0<\e<1$ be a constant to be determined. Suppose $\Lambda \in \X$. Then using Proposition \ref{prop: contraction varphi}, we get that
    \begin{align}
        &\int_{\Mat} f_{\e,\bfn}(g_tu(x)\Lambda) \, d\mur(x) \nonumber \\
        &= \e^{-1} +  \int_{\Mat} \sum_{l=1}^{d-1}  \varphi_l^{\eta_l }(g_tu(x)\Lambda) \, d\mur(x) \nonumber \\ 
        &\leq \e^{-1}  + C_{\bfn} t^{-\eta} \sum_{l=1}^{d-1} \varphi_l^{\eta_l }(\Lambda)  +  \xi(t) \left( \sum_{l=1}^{d-1}  \max_{1 \leq j \leq \min\{l, d-l\}} \left(  \varphi_{l-j}(\Lambda)\varphi_{l+j}(\Lambda)\right)^{\eta_l/2} \right) \label{eq: w 9}
    \end{align}
    
    Note that 
    \begin{align}
        \varphi_{l-j}(\Lambda) &\leq  f_{\e, \bfn}^{\frac{1}{\eta_{l-j}}}(\Lambda), \label{eq: 1 1 1} \\
        \varphi_{l+j}(\Lambda) &\leq  f_{\e, \bfn}^{\frac{1}{\eta_{l+j}}}(\Lambda), \label{eq: 1 1 2} \\
        1 \leq \left(\e f_{\e, \bfn}(\Lambda) \right)^{1-\frac{\eta_l}{2}\left(\frac{1}{\eta_{l-j}} + \frac{1}{\eta_{l+j} } \right)} &\leq \e^{\an} f_{\e, \bfn}^{1-\frac{\eta_l}{2}\left(\frac{1}{\eta_{l-j}} + \frac{1}{\eta_{l+j} } \right)}(\Lambda). \label{eq: 1 1 3}
    \end{align}
    Thus, we have
    \begin{align}
        \left(  \varphi_{l-j}(\Lambda)\varphi_{l+j}(\Lambda)\right)^{\eta_l/2} &\leq \left( f_{\e, \bfn}^{\frac{1}{\eta_{l-j}}}(\Lambda) f_{\e, \bfn}^{\frac{1}{\eta_{l+j}}}(\Lambda) \right)^{\eta_l/2} \quad \text{ using \eqref{eq: 1 1 1}, \eqref{eq: 1 1 2}}  \nonumber \\
        &= f_{\e, \bfn}^{\frac{\eta_l}{2}\left(\frac{1}{\eta_{l-j}} + \frac{1}{\eta_{l+j} } \right)}(\Lambda)  \nonumber\\
        &\leq f_{\e, \bfn}^{\frac{\eta_l}{2}\left(\frac{1}{\eta_{l-j}} + \frac{1}{\eta_{l+j} } \right)} (\Lambda) \left(\e^{\an} f_{\e, \bfn}^{1-\frac{\eta_l}{2}\left(\frac{1}{\eta_{l-j}} + \frac{1}{\eta_{l+j} } \right)}(\Lambda)) \right) \nonumber     \quad \text{ using \eqref{eq: 1 1 3}}\\
        &\leq \e^{\an} f_{\e, \bfn}(\Lambda). \label{eq: w 10}
    \end{align}
    Thus, we get from \eqref{eq: w 9} and \eqref{eq: w 10} that 
    \begin{align*}
        \int_{\Mat} f_{\e,\bfn}(g_tu(x)\Lambda) \, d\mur(x)  \leq C_{\bfn} t^{-\eta} f_{\e,\bfn}(\Lambda) + \e^{-1}(1- C_{\bfn}t^{-\eta}) + (d-1) \e^{\an} \xi(t) f_{\e, \bfn}(\Lambda)
    \end{align*}
    Choose $\e$ so that $\e^{\an}= \frac{C_{\bfn} t^{-\eta}}{(d-1) \xi(t)}$ and $b= \e^{-1}(1- C_{\bfn}t^{-\eta })$, we get \eqref{eq: height fn contraction}. Thus, the proposition follows.
\end{proof}


\section{The Contraction Hypothesis}
\label{sec: The Contraction Hypothesis}

\begin{defn} [The Contraction Hypothesis]
      \label{def: Contraction Hypothesis}
      Suppose $Y$ is a metric space equipped with a continuous action of $G$. Given a collection of functions $\{ f_\tau: Y \rightarrow [0,\infty]: \tau \in S\}$ for some unbounded set $S \subset (0,\infty)$ and $\beta > 0$, we say that $\mu$ satisfies the $((f_\tau)_\tau, \beta)$-\textbf{contraction hypothesis} on $Y$ if the following properties hold: 

\begin{enumerate}
    \item The set $Y_f = \{y \in Y : f_\tau(y) = \infty\}$ is independent of $\tau$ and is $G$-invariant.

    \item For every $\tau \in S$, $f_\tau$ is uniformly log-Lipschitz with respect to the $G$-action. That is, for every bounded neighborhood $\mc{O}$ of the identity in $G$, there exists a constant $C_\mc{O} \geq 1$ such that for all $h \in \mc{O}$, $z \in Y$, and $\tau \in S$,
    \begin{align}\label{eq: assum log lipschitz f t}
        C_\mc{O}^{-1} f_\tau(z) \leq f_\tau(h z) \leq C_\mc{O} f_\tau(z).
    \end{align}

    \item There exists a constant $c \geq 1$ such that the following holds: for every $\tau \in S$, there exists $T > 0$ such that for all $z \in Y$, $r \in \Xi$, and $f_\tau(z) > T$,
    \begin{align}\label{eq: assum Contraction f t}
        \int_{\Mat} f_\tau(g_\tau u(x) z) \, d\mur(x) \leq c f_\tau(z) \tau^{-\beta}.
    \end{align}
\end{enumerate}

The functions $f_\tau$ will be referred to as \textbf{height functions}.
\end{defn}
\begin{rem}
The above definition of the contraction hypothesis is motivated by the corresponding definition in \cite{Khalilsing}.
\end{rem}

\begin{defn}
\label{def:div}
   Suppose $Y$ is a locally compact second countable metric space equipped with a continuous $G$ action. Given a closed $G$-invariant subset $Y' \subset Y$, $0 < p \leq 1$ and $y \in Y \setminus Y'$, we define $\Div(y,Y', p)$ as the set of all $x \in \Mat$ such that
   $$
    \liminf_{T \rightarrow \infty } \frac{1}{T} \int_0^T \delta_{g_{e^t}u(x)y}(Y \setminus K) \, dt \geq p,
   $$
   for all compact subsets $K \subset Y \setminus Y'$.
\end{defn}

\begin{lem}
\label{lem: discretisation}
    Suppose $Y$ is a locally compact second countable metric space equipped with a continuous $G$ action. Given a closed $G$-invariant subset $Y' \subset Y$, $0 <q< p \leq 1$ and $y \in Y \setminus Y'$, if $x \in \Div(y, Y',p)$, then for all compact subsets $K \subset Y \setminus Y'$ and $t>1$, there exists $N$ such that for every $M>N$, we have
    $$
        \frac{1}{M}\#\{l \in [1,M]\cap \N: g_{t}^lu(x)y \notin K \} \geq q.
    $$
\end{lem}
\begin{proof}
    Fix $x \in \Div(y, Y',p)$ and a compact set $K \subset Y \setminus Y'$. Note that it is enough to show that  
    \begin{align}
        \label{eq: a b c d 1}
        \limsup_{M \rightarrow \infty} \frac{1}{M}\#\{l \in [1,M]\cap \N: g_{t}^lu(x)y \in K \} \leq 1- p.
    \end{align}
    To prove \eqref{eq: a b c d 1}, define $K^t= \cup_{\tau=t^{-1}}^1 g_{\tau} K$. Clearly if $g_t^lu(x)y \in K$, then $g_\tau u(x)y \in K^t$ for all $\tau \in [t^{l-1}, t^{l}]$. Thus, we have
    \begin{align*}
        &\limsup_{M \rightarrow \infty} \frac{1}{M}\#\{l \in [1,M]\cap \N: g_{t}^lu(x)y \in K \} \\
        & \leq  \limsup_{M \rightarrow \infty} \frac{1}{M} \sum_{l=1}^M \frac{1}{\log t} \int_{(l-1) \log t}^{l \log t} \delta_{g_{e^\tau} u(x)y}(K^t)\, d\tau \\
        &\leq \limsup_{T \rightarrow \infty } \frac{1}{T} \int_0^T \delta_{g_{e^\tau}u(x)y}( K^t) \, d\tau \\
        &\leq 1-p,
    \end{align*}
    where last inequality follows from the definition of $\Div(y, Y', p)$. Hence \eqref{eq: a b c d 1} holds. This proves the lemma.
\end{proof}

The main aim of this section is to prove the following theorem.
\begin{thm}
\label{thm: contraction implies dimesnion bound}
    	Let $Y$ be a locally compact second countable metric space equipped with a continuous action of $G$. Assume that there exists a collection of continuous functions $\{ f_\tau: Y \rightarrow [0,\infty]: \tau \in S\}$ for some unbounded set $S \subset (0,\infty)$ and $0<\beta< (a_1+b_1)s $, such that $\mu$ satisfies the $(\{f_\tau\}_{\tau \in S}, \beta)$-contraction hypothesis on $Y$. Assume that $Y_f = \{y \in Y : f_\tau(y) = \infty\}$, which is independent of $\tau$ and is $G$-invariant. Then for all $y\in Y \backslash Y_f$ and $0<p\leq 1$, 
        \begin{align}
        \label{eq: cont implies dimension bound}
           \dim_P \left( \Div(y,Y_f, p) \cap \Kcal \right) \leq s-\frac{p\b}{a_1+b_1}.
        \end{align}
        Also, for any sequence $(c_{\tau})_{\tau \in S}$ of positive real numbers and $0<a \leq (a_1+ b_1)s- \beta$, we have
        \begin{align}
        \label{eq: cont implies dimension bound 2}
           \dim_P \left(x\in \mc{K}: \substack{\text{ for all $\tau \in S$, the following holds for all sufficiently large $t$} \\  f_\tau(g_tu(x)y) \geq c_\tau t^{a} }  \right) \leq s-\frac{a+\beta }{a_1+b_1} .
        \end{align}
\end{thm}

\begin{proof}
    Fix $y \in Y$ and $q \in(0,p)$. Without loss of generality, we may assume $0 \in \Kcal$ by replacing $y$ with $u(x)y$ for some $x \in \Kcal$. 

    Let us briefly recall some notation related to the fractal $\Kcal$. For $1 \leq i \leq m$ and $1 \leq j \leq n$, the set $\Kcal_{ij}$ is the limit set of the IFS $\Phi_{ij} = \{\phi_{ij,e} : e \in E_{ij}\}$, with a common contraction ratio $c_{ij}$ and cardinality $p_{ij} = \#E_{ij}$. The dimension of $\Kcal_{ij}$ is given by $s_{ij} = -\log p_{ij} / \log c_{ij}$. The set $\Kcal = \prod_{ij} \Kcal_{ij} \subset \Mat$ has dimension $s = \sum_{ij} s_{ij}$. 

    The measure $\mu_{ij}$ denotes the normalized restriction of $H^{s_{ij}}$ to $\Kcal_{ij}$, and the measure $\mu$ on $\Kcal$ is defined as $\mu = \otimes_{ij} \mu_{ij}$. For $1 \leq i \leq m$, $1 \leq j \leq n$, and $l \in \N$, we define $\Fcal_{ij}(l)$ as in \eqref{eq: def K w, collection} corresponding to $\Phi_{ij}$. We also define $\sigma_{ij}: E_{ij}^\N \rightarrow \R$ as in \eqref{eq:def sigma}. Furthermore, we define $\kappa_{ij}: \cup_l \Fcal_{ij}(l) \rightarrow \R$ by 
    \[
    \kappa_{ij} \big(\phi_{ij,\bfe_1} \circ \cdots \circ \phi_{ij,\bfe_l}(\Kcal_{ij})\big) = \phi_{ij,\bfe_1} \circ \cdots \circ \phi_{ij,\bfe_l}(0).
    \]

    Fix $\alpha > 0$ large enough such that $\bigcup_{r \in \Xi} \supp(\mur) \subset \M_{m \times n}([-\alpha, \alpha])$, and define $O \subset G$ by 
    \begin{align}
        \label{eq: w 17}
        O = \{u(x) : x \in \M_{m \times n}([-2\alpha, 2\alpha])\}.
    \end{align}
    Let $A > 1$ be a constant such that \eqref{eq: assum log lipschitz f t} holds for all $\tau \in S$, $z \in Y$, and $h \in O$ with $C_O = A$. Fix $t \in S$ sufficiently large such that $t^{a_i + b_j} > c_{ij}^{-1}$ for all $i$. Let $T' > 0$ and $c > 0$ be constants such that \eqref{eq: assum Contraction f t} holds for $f_t$ and all $z \in Y$ with $f_t(z) > T'$. Suppose $B_t>1$ is a constant such that \eqref{eq: assum log lipschitz f t} holds for all $\tau \in S$, $z \in Y$, and $h \in \{g_su(x): s \in [t^{-1},t], x \in \M_{m \times n}([-2\alpha, 2\alpha])\}$ with $C_O = B_t$. Define $T = \max\{f_t(y), A^2B_tT'\}$. For simplicity, let $g = g_t$. 

    For every $k \in \N$, $1 \leq i \leq m$, and $1 \leq j \leq n$, define $N_k(i,j)$ as the unique integer such that 
    \[
    c_{ij}^{N_k(i,j) + 1} < t^{-k(a_i + b_j)} \leq c_{ij}^{N_k(i,j)}.
    \]
    Next, define $\Fcal(k) = \prod_{ij} \Fcal_{ij}(N_k(i,j))$ and $\kappa: \cup_k \Fcal(k) \rightarrow \Mat$ as the restriction of $\prod_{ij} \kappa_{ij}$. Since $0 \in \Kcal$, it follows that $\kappa(R) \in R$ for all $R \in \cup_k \Fcal(k)$.  For notational convenience, we also set $\Fcal(0)= \{\Kcal\}$ and $\kappa(\Kcal)=0$.
    Define the element $\bfr_k \in \Xi$ for all $k \in \N$ by 
    \begin{align}
        \label{eq: w 13}
        (\bfr_k)_{ij} = t^{k(a_i + b_j)}c_{ij}^{N_k(i,j)} \in [1, c_{ij}^{-1}].
    \end{align}
    Finally, for all $r \in \Mat$, define $\psi_r: \Mat \rightarrow \Mat$ by 
    \begin{align}
        \label{eq: def psi}
        \psi_r((\theta_{ij})) = (r_{ij}\theta_{ij}).
    \end{align}

    Note that if $x \in \Div(y,Y_f,p)$, then by Lemma \ref{lem: discretisation}, we know that for every compact subset $K$ of $Y \setminus Y_f$, there exists $N$ such that for every $M>N$, we have
    \begin{align}
        \label{eq: ww 15}
         \frac{1}{M}\#\{l \in [1,M]\cap \N: g_{t}^lu(x)y \notin K \} \geq q.
    \end{align}
    Using continuity of $f_t$, we know that $f_t^{-1}([0,T])$ is a compact subset of $Y \setminus Y_f$, and hence using \eqref{eq: ww 15} for $K= f_t^{-1}([0,T])$, we get that 
    \begin{align}
    \label{eq: w 15}
         \Div(y,Y_f,p) \cap \Kcal &\subset \bigcup_{N \in \N  }  Z(N), 
    \end{align}
    where
    \begin{align*}
          Z(N) &= \left\{x \in \Kcal: \frac{1}{M}\#\{l \in [1,M]\cap \N: f_t(g^l u(x)y)> T \} \geq q \text{ for all } M \geq N \right\}.
    \end{align*}
Similarly,
    \begin{align}
    \label{eq: w 15'}
         \left\{x\in \mc{K}: \substack{\text{ for all $\tau \in S$, the following holds for all large enough $s$} \\  f_\tau(g_su(x)y) \geq c_\tau s^{a} } \right\} &\subset \bigcup_{N \in \N  }  Z'(N),
    \end{align}
    where 
    \begin{align*}
      , \\
        Z'(N) &= \{x \in \Kcal: f_t(g^M u(x)y)> \max\{T,c_t t^{Ma}\} \text{ for all } M \geq N\}.
    \end{align*}
    Also define for $i<j$, the set
    $$B(i,j)= \{x \in \Kcal: f_t(g^l u(x)y)> T \text{ for all } i < l \leq j\}.$$
     Now fix $N \in \N$.  Before proceeding further, let us make some easy observations: \\

       \noindent {\bf Observation 1:} For $M \in \N$ and $K_1, K_2 > 0$, assume that $x \in R \in \Fcal(M)$ satisfies $K_1 \leq f_t(g^M u(x)y) \leq K_2$. Then for all $x' \in R$, we have $ K_1/A \leq f_t(g^M u(x')y) \leq AK_2$. 

      \noindent {\bf Explanation:} If $x' \in R$, the $(i,j)$-th coordinate of $x' - x$ is less than $c_{ij}^{N_M(i,j)} = (\bfr_M)_{ij} t^{-M(a_i+b_j)}$ times the diameter of $\supp(\mu_{ij})$, which equals $t^{-M(a_i+b_j)}$ times the diameter of $\supp(\mu_{ij}^{(\bfr_M)})$. This is further less than $2t^{-M(a_i+b_j)} \alpha$. Define $x'' \in \Mat$ as the vector whose $(i,j)$-th entry is $t^{M(a_i+b_j)}$ times the $(i,j)$-th entry of $x' - x$. Then $u(x'') \in O$, and we obtain
       \[
        f_t(g^M u(x')y) = f_t(g^M u(x' - x) g^{-M} g^M u(x)y) = f_t(u(x'') g^M u(x)y) \begin{cases}
            \geq \frac{f_t(g^M u(x)y)}{A} \geq \frac{K_1}{A}. \\
            \leq Af_t(g^M u(x)y) \leq AK_2.
        \end{cases}
        \]
        Thus, the observation follows. \\

        \noindent {\bf Observation 2:} For all $  i< l \leq j$, the following holds: if $R \in \Fcal(l)$ satisfies $R \cap B(i,j) \neq \emptyset$, then $f_t(g^l u(x)y) > T/A$ for all $x \in R$. Similarly for all $M>N$, if $R \in \Fcal(M)$ satisfies $R \cap Z'(N) \neq \emptyset$, then $f_t(g^M u(x)y) > \max\{ T, c_t t^{M a}\}/A$ for all $x \in R$. 

        \noindent {\bf Explanation:} This observation follows directly from the definitions of $B(i,j)$, $Z'(N)$, and from Observation 1. \\

        \noindent {\bf Observation 3} For any measurable function $h: \Mat \rightarrow \R_+$, measurable set $X \subset \Mat$ and $M \in \N$, we have
        $$
        \sum_{\substack{R \in \Fcal(M) \\ R \cap X \neq \emptyset} } \int_R h(x)\, d\mu(x) \leq \sum_{\substack{R \in \Fcal(M-1) \\ R \cap X \neq \emptyset} } \int_R h(x)\, d\mu(x).
        $$
        \noindent {\bf Explanation:} Using the fact that each $(\Psi_{ij})$ satisfies the open set condition, we get from Proposition \ref{prop: important IFS} that $\mu(R \cap R')=0$ for all $R \neq R' \in \Fcal(M)$. Thus for any $R \in \Fcal(M-1)$, the following holds  
        $$
        \int_R h(x) \, d\mu(x) = \sum_{\substack{R' \in \Fcal(M) \\ R' \subset R}} \int_{R'} h(x) \, d\mu(x).
        $$
        The observation now follows immediately. \\

        \noindent {\bf Observation 4} For any measurable function $h: Y \rightarrow \R_+$, $M \in \N$ and $R \in \Fcal(M)$, $y \in Y$, we have
        $$
        \int_R h(g^M u(x) y) \, d\mu(x) = \mu(R) \int_{\Mat} h(u(x) g^M u(\kappa(R)) y)\, d\mu^{(\bfr_M)}(x).
        $$
        \noindent {\bf Explanation:} Let us define the matrix $\bfs_M \in \Mat$ as
        \begin{align}
        \label{eq: def bfs}
            (\bfs_M)_{ij}= c_{ij}^{N_{M}(i,j)}.
        \end{align}
        Using the definition of $\Fcal(M)$ and $\kappa$, it is easy to see that $R = \psi_{\bfs_M}(\Kcal) + \kappa(R)$, where $\psi_s$ is defined as in \eqref{eq: def psi}. Since $\mu = \otimes_{ij} \mu_{ij}$ and $\mu_{ij}$ is a Bernoulli measure using the Proposition \ref{prop: important IFS} corresponding to the uniform measure on $E_{ij}$, we get that $\frac{1}{\mu(R)}\mu|_{R}$ equals pushforward of $\mu$ under the map $x \mapsto\psi_{\bfs_M}(x) + \kappa(R)$. Also, note that for all $ x\in \Mat$, we have $g^Mu(\psi_{\bfs_M}(x)) g^{-M}= u(\psi_{\bfr_M}(x))$. Thus, we have
        \begin{align*}
             \int_R h(g^M u(x) y) \, d\mu(x) &= \mu(R) \int_{\Mat} h(g^M u( \psi_{\bfs_M}(x) + \kappa(R)) y) \, d\mu(x) \\
             &= \mu(R) \int_{\Mat} h(g^M u(\psi_{\bfs_M}(x)) g^{-M} g^M  u(\kappa(R) ) y) \, d\mu(x) \\
             &= \mu(R) \int_{\Mat} h( u(\psi_{\bfr_M}(x)) g^M  u(\kappa(R) ) y) \, d\mu(x) \\
             &=\mu(R) \int_{\Mat}  h( u(x) g^M  u(\kappa(R) ) y) \, d\mu^{(\bfr_M)}(x).
        \end{align*} \\

        \noindent {\bf Observation 5} For any $M \in \N$ and $R \in \Fcal(M)$, $y \in Y$, we have
        $$
        \mu(R) f_t(g^M u(\kappa(R)) y) \leq A \int_{R} f_t( g^M u(x) y)\, d\mu(x).
        $$
        \noindent {\bf Explanation:} Note that 
        \begin{align*}
            \mu(R)f_t(g^M u(\kappa(R)) y) &\leq \mu(R) A \int_{\Mat}f_t(  u(x) g^M u(\kappa(R)) y) \, d\mu^{(\bfr_M)}(x) \quad \text{using the definition of $A$} \\
            &= A \int_{R} f_t( g^M u(x) y) \, d\mu(x) \quad (\text{ using Observation 4}).
        \end{align*} \\

        \noindent {\bf Observation 6} For $ i < l \leq j$ and $R \in \Fcal(l-1)$ such that $R \cap B(i,j) \neq \emptyset$, we have
        $$
            \int_{\Kcal}f_t(g u(x) g^{l-1} u(\kappa(R)) y) \, d\mu^{\bfr_{l-1}}(x) \leq ct^{-\beta} \mu(R) f_t(g^{l-1}u(\kappa(R)) y).
        $$
        \noindent {\bf Explanation:} Since $R \cap B(i,j) \neq \emptyset$, we get a point $x' \in R$ such that $f_t(g^lu(x')y)>T$. Using definition of $B_t$, we get that $f_t(g^{l-1}u(x')y)>T/B_t$. From Observation 1 and the fact that $\kappa(R) \in R$ that $f_t(g^{l-1}u(\kappa(R)) y)> T/AB_t=T'$. Now the above observation follows by using \eqref{eq: assum Contraction f t}. \\

\noindent {\bf Observation 7} For all $ i<j$ and subset $Z \subset B(i,j)$, we have
    \begin{align*}
        \sum_{\substack{R \in \Fcal(j) \\ R \cap Z \neq \emptyset}}\int_{R} f_t(g^{j}u(x)y) \, d\mu(x) \leq (cAt^{-\beta})^{j-i}\sum_{\substack{R \in \Fcal(i) \\ R \cap Z \neq \emptyset}}\int_{R} f_t(g^{i}u(x)y) \, d\mu(x)
    \end{align*}
    \noindent {\bf Explanation} 
    Note that for any $i < l \leq j$, we have
          \begin{align}
        &\sum_{\substack{R \in \Fcal(l) \\ R \cap Z \neq \emptyset}} \int_{R} f_t(g^l u(x)y) \, d\mu(x)\nonumber \\
        &\leq  \sum_{\substack{R \in \Fcal(l-1) \\ R \cap Z \neq \emptyset}} \int_{R} f_t(g^l u(x)y) \, d\mu(x) \quad (\text{ using Observation 3})  \nonumber\\
        &=  \sum_{\substack{R \in \Fcal(l-1) \\ R \cap Z \neq \emptyset }} \mu(R) \int_{\Mat} f_t(g u(x) g^{l-1} u(\kappa(R)) y) \, d\mu^{(\bfr_{l-1})}(x) \quad (\text{ using Observation 4}) \nonumber \\
        &\leq  ct^{-\beta}  \sum_{\substack{R \in \Fcal(l-1) \\ R \cap Z \neq \emptyset }} \mu(R) f_t(g^{l-1} u(\kappa(R)) y)  \quad (\text{ using Observation 6}) \nonumber\\
        &\leq cA t^{-\beta}  \sum_{\substack{R \in \Fcal(l-1) \\ R \cap Z \neq \emptyset } }  \int_{R} f_t(g^{l-1} u(x) y) \, d\mu(x) \quad (\text{ using Observation 5}). \label{eq: w 19}
    \end{align}
    Iteratively using \eqref{eq: w 19} for $l=j,j-1, \ldots, i+1$,the observation follows. \\

     \noindent {\bf Observation 8} Suppose $Z \subset \Mat$ and $\epsilon>0$ is given such that for all $M >N$, we have
        $$
        \sum_{\substack{R \in \Fcal(M) \\ R \cap Z \neq \emptyset}} \mu(R) \leq C_Z B^{M} t^{-M\epsilon},
        $$
        for some constants $B, C_Z$ independent of $M$. Then 
        $$
        \dim_P(Z \cap \Kcal) \leq s- \frac{1}{a_1+b_1} \left(\epsilon -  \frac{\log B}{\log t} \right).
        $$
        \noindent {\bf Explanation:} We will use Lemma \ref{lem: Falconer} to prove the observation. However, we will need the following notation. For $1 \leq i \leq m$, $1 \leq j \leq n$ and $l \in \N$, we define $P_l(i,j)$ as the unique integer satisfying 
        $$
        c_{ij}^{P_l(i,j)} \leq t^{-(a_1+b_1)l}< c_{ij}^{P_l(i,j)-1}. 
        $$
        Note that $a_1 + b_1 = \max_ia_i + \max_j b_j$, hence $P_l(i,j) \geq N_l(i,j)$.  Let $\delta>0$ be given. Assume that $M_\delta>0$ is such that $2\alpha t^{-M_\delta(a_1 +b_1)} \leq \delta < 2\alpha t^{-(M_\delta-1)(a_1+ b_1)}$. Assume that $\delta$ is small enough so that $M_\delta >N$. Let us try to cover $Z$ by sets of diameter less than or equal to $\delta$. We do this by selecting sets in $\prod_{ij} \Fcal_{ij}(P_{M_\delta}(i,j))$ which intersect $Z$. Note that the total number of elements in $\prod_{ij} \Fcal_{ij}(P_{M_\delta}(i,j))$ equals  
        $$
        \prod_{i,j} p_{i,j}^{P_{M_\delta}(i,j)}= \prod_{i,j}c_{i,j}^{-s_{i,j} P_{M_\delta}(i,j)} \leq \frac{1}{\prod_{i,j}c_{ij}} t^{s(a_1+b_1){M_\delta}},
        $$ 
        and each of them has equal $\mu$-measure. Thus, we get that
    \begin{align*}
        \#\{R \in \prod_{ij} \Fcal_{ij}(P_{M_\delta}(i,j)): R\cap Z \neq \emptyset \} &= (\#\prod_{i,j} \Fcal_{ij}(P_{M_\delta}(i,j))). \left(\sum_{\substack{R \in \prod_{i,j} \Fcal_{ij}(P_{M_\delta}(i,j)) \\ R\cap Z \neq \emptyset }} \mu(R) \right),\\
        &\leq \frac{1}{\prod_{i,j}c_{ij}} t^{s(a_1+b_1){M_\delta}}. \left(\sum_{\substack{R \in \Fcal({M_\delta}) \\ R\cap Z(N) \neq \emptyset }} \mu(R) \right), \\
        &\leq C_Z' \left(B t^{-\epsilon} t^{s(a_1+b_1)}\right)^{M_\delta},
    \end{align*}
    where $C_Z'= C_Z/(\prod_{i,j}c_{ij})$. Thus, we can cover $Z$ by $C_Z' \left(B t^{-\beta} t^{s(a_1+b_1)}\right)^{M_\delta}$-many sets of diameter less than or equal to $\max_{i,j}\{c_{ij}^{P_{M_\delta}(i,j)} \diam(\Kcal_{i,j})\} \leq 2\alpha. t^{-{M_\delta}(a_1+b_1)} \leq \delta$. Thus we have
    \begin{align}
        \label{eq: 1 3 1}
        \frac{\log L_\delta(Z)}{- \log \delta} \leq \frac{ \log(C_Z') + {M_\delta} \log (B t^{-\beta} t^{s(a_1+b_1) }) }{ -\log(2\alpha) +  (M_\delta-1)(a_1 + b_1) \log t} ,
    \end{align}
    where $L_\delta(Z)$ denotes the smallest number of sets of diameter at most $\delta$ that cover $Z$. Since $M_\delta \rightarrow \infty$ as $\delta \rightarrow 0$, we get from Lemma \ref{lem: Falconer} and \eqref{eq: 1 3 1} that
    \begin{align*}
        \dim_P(Z) &\leq \limsup_{M \rightarrow \infty} \frac{ \log(C_Z') + {M} \log (B t^{-\epsilon} t^{s(a_1+b_1) })}{ -\log(2\alpha) +  (M-1)(a_1 + b_1) \log t}  \\
        &= s - \frac{1}{a_1+ b_1} \left( \epsilon - \frac{\log B}{\log t} \right).
    \end{align*}
    Thus the observation follows. \\

    \noindent {\bf Observation 9} We have 
        \begin{align*}
           \dim_P(Z'(N) \cap \Kcal) \leq s- \frac{1}{a_1+b_1} \left(\beta + a -  \frac{\log A+ \log c}{\log t} \right).
        \end{align*}
        \noindent {\bf Explanation} Fix $M>N$. Note that by Observation 2, we have
         \begin{align}
            \label{eq: w 18}
              \sum_{\substack{R \in \Fcal(M) \\ R \cap Z'(N) \neq \emptyset}} \mu(R) &\leq \frac{A}{c_t t^{Ma}} \sum_{\substack{R \in \Fcal(M) \\ R \cap Z \neq \emptyset}} \int_{R} f_t(g^M u(x)y) \, d\mu(x).
          \end{align}
       Since $Z'(N) \subset B(M,N-1)$, we get from Observation 7 that
       \begin{align}
       \label{eq: w 19'}
        \sum_{\substack{R \in \Fcal(M) \\ R \cap Z \neq \emptyset}}\int_{R} f_t(g^{M}u(x)y) \, d\mu(x) \leq (cAt^{-\beta})^{M+1-N}\sum_{\substack{R \in \Fcal(N-1) \\ R \cap Z \neq \emptyset}}\int_{R} f_t(g^{N-1}u(x)y) \, d\mu(x).
    \end{align}
    Using \eqref{eq: w 18} and \eqref{eq: w 19'}, we get that for any $M>N$ the following holds
    \begin{align}
        \label{eq: w 20}
        \sum_{\substack{R \in \Fcal(M) \\ R \cap Z'(N) \neq \emptyset}} \mu(R) &\leq C_N c^M A^{M} t^{-M(\beta+ a)},
    \end{align}
     where $C_N= c^{1-N}A^{2-N}t^{\beta (N-1)} \int_{\Kcal} f_t(g^{N} u(x) y) \, d\mu(x)/ c_t$. The observation now follows from Observation 8. \\

    \noindent {\bf Observation 10} We have 
        \begin{align*}
           \dim_P(Z(N) \cap \Kcal) \leq s- \frac{1}{a_1+b_1} \left(q\beta -  \frac{\log2 + 3\log A + \log c }{\log t} \right).
        \end{align*}
    
    \noindent {\bf Explanation} Fix $M>N$. Fix a subset $Q\subset \{1, \ldots, M\}$ containing at least $qM$ elements. Let us define 
    $$
    Z(N,M,Q)= \{x \in Z(N): \text{ For $1 \leq l \leq M$},  f_t(g^lu(x)y)>T \text{ iff } l \in Q\}.
    $$
     We claim that
    \begin{align}
        \label{eq: measure of z m n q}
        \sum_{\substack{R \in \Fcal(M) \\ R \cap Z(N,M,Q) \neq \emptyset}} \mu(R) \leq c^{\#Q}t^{-\#Q \beta} A^{\#Q+ 2M} \leq c^Mt^{-qM \beta}A^{3M}.
    \end{align}
    To show this claim, we decompose the set $Q$ and its complement into maximal connected intervals as
    $$
    Q= \bigsqcup_{i=1}^{l_1} I_i', \quad \{1, \ldots, M\} \setminus Q= \bigsqcup_{i=1}^{l_2} I_i'',
    $$
    for some $l_1, l_2\geq 0$. Note that $|l_1-l_2|\leq 1$. Order the intervals $I_i'$ and $I_j''$ in the way they appear in the sequence $1 \leq \cdots \leq M$. Write $I_i$ for the $i$-th interval in this sequence for $1 \leq i \leq l_1 +l_2$. Claim that if $I_i= \{j_1, \ldots, j_2\}$, then 
    \begin{align}
        \label{eq: measure induct 1}
        \sum_{\substack{R \in \Fcal(j_2) \\ R \cap Z(N,M,Q) \neq \emptyset}} \mu(R) \leq \begin{cases}
         c^{\# I_i}   t^{-\#I_i \beta} A^{\#I_i + 2} \sum_{\substack{R \in \Fcal(j_1-1) \\ R \cap Z(N,M,Q) \neq \emptyset}} \mu(R) \text{ if } I_i \subset Q, \\
            \sum_{\substack{R \in \Fcal(j_1-1) \\ R \cap Z(N,M,Q) \neq \emptyset}} \mu(R) \text{ otherwise. } 
        \end{cases}
    \end{align}
    To prove this note that if $I_i \subset Q$, then we have $Z(N,M,Q) \subset B(j_1-1, j_2)$. Thus, we have
    \begin{align*}
        &\sum_{\substack{R \in \Fcal(j_2) \\ R \cap Z(N,M,Q) \neq \emptyset}} \mu(R) \\
        &\leq \frac{A}{T} \sum_{\substack{R \in \Fcal(j_2) \\ R \cap Z(N,M,Q) \neq \emptyset}} \int_R f_t(g^{j_2}u(x)y) \, d\mu(x) \text{ using Observation 2}\\
        &\leq \frac{A}{T} (cAt^{-\beta})^{j_2+1-j_1} \sum_{\substack{R \in \Fcal(j_1-1) \\ R \cap Z(N,M,Q) \neq \emptyset}} \int_R f_t(g^{j_1-1}u(x)y) \, d\mu(x) \text{ using Observation 7}\\
        &\leq \frac{A}{T} (cAt^{-\beta})^{j_2+1-j_1}  \sum_{\substack{R \in \Fcal(j_1-1) \\ R \cap Z(N,M,Q) \neq \emptyset}} AT \mu(R) .
    \end{align*}
    To justify the last inequality, it suffices to show that
$f_t(g^{j_1 - 1} u(x) y) < AT$ for all $x \in R\in \Fcal(j_1-1)$ with $R \cap Z(N,M,Q) \neq \emptyset$. If $j_1 = 1$, this follows from the bound $f_t(y) \leq T$ along with \eqref{eq: assum log lipschitz f t} which holds for $C_0= A$ and all $h =u(x)$ for $x \in \supp(\mu)$. For $j_1 > 1$, note that $j_1 - 1 \notin Q$. Hence, for every $R \in \Fcal(j_1 - 1)$ with $R \cap Z(N,M,Q) \neq \emptyset$,
there exists $x' \in R$ such that $f_t(g^{j_1-1}u(x')y)< T$, which combined with Observation 1 implies that $f_t(g^{j_1-1}u(x)y)< AT$ for all $x \in R$. This proves \eqref{eq: measure induct 1} for case $I_i \subset Q$. The other case is trivial. Using \eqref{eq: measure induct 1} iteratively for $I_{l_1+ l_2}, \ldots, I_1$, the \eqref{eq: measure of z m n q} follows.

    Note that for any $M>N$, we have $Z(N)= \cup_Q Z(N,M,Q)$, where sum is taken over all $Q$ such that $\#Q \geq qM$. Thus using \eqref{eq: measure of z m n q} and the fact that there are at most $2^M$ subsets of $\{1, \ldots, M\}$, we get
    \begin{align}
         \sum_{\substack{R \in \Fcal(M) \\ R \cap Z(N) \neq \emptyset}} \mu(R) \leq 2^Mt^{-qM\beta} A^{3M}c^M.
    \end{align}
    The observation now follows by using Observation 8. \\

 Note that $\dim_P(\cup_l I_l) = \sup_l \dim_P(I_l)$ for any countable collection of Borel sets $I_l$. Thus, from \eqref{eq: w 15} and Observation 10  (respectively \eqref{eq: w 15'}, Observation 9), we get that
  \begin{align*}
       \dim_P(\Div(y,Y_f,p) \cap \Kcal) \leq s- \frac{1}{a_1+b_1} \left(q\beta -  \frac{\log2+ 3\log A + \log c}{\log t} \right) \\
        \dim_P \left(x\in \mc{K}: \substack{\text{ for all $\tau \in S$, the following holds for all large enough $t$} \\  f_\tau(g_tu(x)y) \geq c_\tau t^{a} }  \right) \leq  s- \frac{1}{a_1+b_1} \left(a+\beta -  \frac{\log A + \log c}{\log t} \right)
  \end{align*}
  Since $A$ and $c$ are independent of $t$ and $q$, letting $q \rightarrow p$ and $t \rightarrow \infty$ (which exists as $S$ is unbounded), we get that the theorem holds. Hence proved.
\end{proof}

\section{Final Proof}

\begin{lem}
\label{lem: Sing Dynamical Interpretation}
    If $\theta \in \Sing(a,b)$, then   $$ \varphi_1(g_t u(\theta)\Gamma) \rightarrow \infty$$ as $t \rightarrow \infty$. As a result, $\Sing(a,b) \subset \Divergent(\Gamma, 1)$.
\end{lem}
\begin{proof}
    By definition, $\theta  \in \Sing(a,b)$ if, for every $0< \delta <1$, there exists $T_{\delta}>0$ such that for all $t > T_{\delta}$, there exist integers $(p, q) \in \Z^m \times (\Z^n \setminus \{0\})$ with the vector $z = (p + \theta q, q) \in u(\theta) \Z^d$ satisfying:
\begin{align*}
    |z_i|^{1/a_i} &\leq \frac{\delta}{t} \quad \text{for all } 1 \leq i \leq m, \\
    |z_{j+m}|^{1/b_j} &\leq t \quad \text{for all } 1 \leq j \leq n.
\end{align*}

    This means that for $\tau = \delta^{ -a_m/(a_m+b_n)} t$, we have that $g_{\tau}u(\theta)\Z^d$ contains the vector $g_{\tau}z = (z_1',\ldots, z_d')$, which satisfies
    \begin{align*}
        |z_i'| = \delta^{ -a_ia_m/(a_m+b_n)} t^{a_i}|z_i| \leq \delta^{ -a_ia_m/(a_m+b_n)} t^{a_i} \delta^{a_i}t^{-a_i}  \leq \delta^{a_mb_n/(a_m+b_n)}  \quad \text{for all } 1 \leq i \leq m \\
        |z_{j+m}'| = \delta^{ b_ja_m/(a_m+b_n)} t^{-b_j}|z_{j+m}| \leq \delta^{ b_ja_m/(a_m+b_n)} t^{-b_j}t^{b_j}  \leq \delta^{a_mb_n/(a_m+b_n)} \quad \text{for all } 1 \leq j \leq n,
    \end{align*}
    where we have used the fact that $a_m = \min_ia_i$ and $b_n = \min_j b_j$. This means that 
    \begin{align}
    \label{eq: t 1}
        \varphi_1(g_\tau u(\theta)\Gamma) \geq \delta^{-a_mb_n/(a_m+b_n)}.
    \end{align}
 Note that \eqref{eq: t 1} holds for all $\tau>T_\delta \delta^{ -a_m/(a_m+b_n)} $. Since $0<\delta<1$ is arbitrary, this lemma follows.
\end{proof}

\begin{lem}
\label{lem: omega sing dynamical Interpretation}
    Let $ 0<\omega < \omega'$. If $\theta \in \VSing(a,b,\omega')$, then there exists $T_\theta:= T_\theta(\omega)$ such that for all $\e>0$ and $t>T_\theta$, we have 
    $$\varphi_1(g_tu(\theta)\Gamma) > t^{\frac{ a_mb_n\omega}{a_m+b_n+ a_m\omega}}. $$
\end{lem}
\begin{proof}
    By definition, if $\theta \in \VSing(a,b,\omega')$, then there exists $T_\theta \geq 1$ such that for all $t>T_\theta$, there exist integers $(p,q) \in \Z^m \times (\Z^n \setminus \{0\})$ with the vector $z= (p+\theta q, q) \in u(\theta) \Z^d$ satisfying:
    \begin{align*}
        |z_i|^{1/a_i} \leq \frac{1}{t^{1+\omega}} \quad \text{for all } 1 \leq i \leq m \\
        |z_{j+m}|^{1/b_j} \leq t \quad \text{for all } 1 \leq j \leq n.
    \end{align*}
    This means that for $\tau = t^{1+ \frac{a_m\omega}{a_m+b_n}}$, we have $g_{\tau}u(\theta)\Z^d$ contains the vector $g_{\tau}z = (z_1',\ldots, z_d')$, which satisfies
    \begin{align*}
        |z_i'| &= t^{a_i\left(1+ \frac{a_m\omega}{a_m+b_n} \right)}|z_i| \leq \frac{t^{a_i\left(1+ \frac{a_m\omega}{a_m+b_n} \right)}}{t^{a_i+ a_i\omega}} \leq t^{-\frac{a_mb_n\omega}{a_m+b_n}}  \quad \text{for all } 1 \leq i \leq m \\
        |z_{j+m}'| &= t^{-b_j\left(1+ \frac{a_m\omega}{a_m+b_n} \right) }|z_{j+m}| \leq t^{-b_j\left(1+ \frac{a_m\omega}{a_m+b_n} \right) }t^{b_j} \leq t^{-\frac{a_mb_n\omega}{a_m+b_n}}  \quad \text{for all } 1 \leq j \leq n,
    \end{align*}
    where we used the fact that $a_m = \min_ia_i$ and $b_n = \min_j b_j$. This means that 
    \begin{align}
    \label{eq: t 2}
         \varphi_1(g_\tau u(\theta)\Gamma) \geq t^{ \frac{ a_mb_n\omega }{a_m+b_n}} = \tau^{\frac{ a_mb_n\omega}{a_m+b_n+ a_m\omega}}.
    \end{align}
 Note that \eqref{eq: t 2} holds for all $\tau>(T_\theta)^{{1+ a_m\omega/(a_m+b_n)}} $. This proves the lemma.
\end{proof}

\begin{prop}
    \label{prop:Genral Estimate}
    Let $\eta, \eta_1, \ldots, \eta_{d-1} \in \R$ be a sequence satisfying the following conditions:
    \begin{align*}
        0<\eta_i &\leq \zeta_i(\mu) \quad \text{for all } 1 \leq i \leq d-1,  \\
        \frac{1}{\eta_{i-j}} + \frac{1}{\eta_{i+j}} &\leq \frac{2}{\eta_i} \quad \text{for all } 1 \leq i \leq d-1, \, 1\leq j \leq \min\{i, d-i\},  \\
        \eta &= \min_{1 \leq l \leq d} w_l \eta_l,
    \end{align*}
    where $1/\eta_0 = 1/\eta_d := 0$. 

   Then, the following bounds hold for all $0<\gamma \leq (s(a_1+ b_1)- \eta)/\eta_1$, $0< p \leq 1$ and $x \in \X$
    \begin{align*} 
        \dim_P( \Divergent(x,p) \cap \Kcal) &\leq s - \frac{p\eta}{a_1 + b_1}, \\ 
        \dim_P( \{ \theta \in \Kcal: \substack{ \text{ there exists $T_\theta >0$ such that for all $t>T_\theta$ ,} \\ \text{ we have $\varphi_1( g_tu(\theta) x) \geq t^{\gamma} $} } \}) &\leq s - \frac{1}{a_1 + b_1} \left( \eta + \eta_1 \gamma \right).
    \end{align*}
\end{prop}

\begin{proof}
We divide the proof into two cases. \\
{\bf Case 1}
    In this case, we assume $\eta, \eta_1, \ldots, \eta_{d-1}$ satisfies following strict inequalities
        \begin{align*}
            \eta_i &< \zeta_i(\mu) \quad \text{ for all } 1 \leq i \leq d-1 \\
           \frac{1}{\eta_{i-j}} + \frac{1}{\eta_{i+j}} &< \frac{2}{\eta_i} \text{ for all } 1 \leq i \leq d-1, 1 \leq j \leq \min\{ i, d-i\}.
        \end{align*}
    In this case, using Proposition \ref{prop: existence of height function}, for every $t>1$, choose $\e(t)$ and define the collection of height functions
$$
 \{f_t:= f_{\e(t), \bfn}: t >1\}.
$$
Now it is easy to see that the action of $G$ on ${\X}$ satisfies $(\{f_t\}, \eta)$-contraction hypothesis with respect to measure $\mu$. Indeed, the first two properties of Definition \ref{def: Contraction Hypothesis} follow immediately from the definition of $f_t$. The third property follows from Proposition \ref{prop: existence of height function} and for $c= 3C_{\bfn}$ and $T= b t^{\eta}/C_{\bfn}$ corresponding to each $t$ (note that the value of $b$ also depends on $t$).

Note that in notation of Definition \ref{def:div}
\begin{align}
    \Divergent(x,p) &= \Div(x, \emptyset, p), \label{eq: c c 1} \\
    \{ \theta \in \Kcal: \substack{ \text{ there exists $T_\theta >0$ such that for all $t>T_\theta$ ,} \\ \text{ we have $\varphi_1( g_tu(\theta) x) \geq t^{\gamma} $} } \} &\subset \{ \theta \in \Kcal: \substack{\text{ there exists $T_\theta >0$ such that for all $t>T_\theta$ and $\tau >1$,} \\ \text{ we have $f_\tau( g_tu(\theta) x) \geq t^{\eta_1 \gamma} $ }} \}. \label{eq: c c 2}
\end{align}

Thus, by Theorem \ref{thm: contraction implies dimesnion bound} and \eqref{eq: c c 1}, \eqref{eq: c c 2}, the proposition follows in this case.

{\bf Case 2} In this case, fix $\eta, \eta_1, \ldots, \eta_{d-1}$, which satisfies the condition of the proposition but does not lie in Case 1.

To proceed further, fix a sequence $q_1, \ldots, q_{d-1}$ satisfying
\begin{align*}
    q_i &>0 \quad \text{ for all } 1\leq i \leq d-1 \\
    2q_i &> q_{i-j} + q_{i+j} \text{ for all } 1 \leq i \leq d-1, 1 \leq j \leq \min\{ i, d-i\},
\end{align*}
where we define $q_0= q_d =0$. The construction of such a sequence is easy.

For every $\delta>0$, we define the sequence $\eta^{(\delta)}, \eta^{(\delta)}_1, \ldots, \eta^{(\delta)}_{d-1}$ as
$$
\eta^{(\delta)}_j = \frac{1}{\frac{1}{\eta_j} + \delta q_j}, \quad \eta^{(\delta)}= \min_{1 \leq l \leq d-1} w_l \eta^{(\delta)}_l,
$$
for all $j$. It is then clear that the sequence $\eta^\delta, \eta^{\delta}_1, \ldots, \eta^{\delta}_{d-1}$ falls in Case 1, which gives
\begin{align*}
    \dim_P( \Divergent(x,p) \cap \Kcal) &\leq s- \frac{p\eta^{(\delta)}}{a_1 + b_1},\\
     \dim_P( \{ \theta \in \Kcal: \substack{ \text{ there exists $T_\theta >0$ such that for all $t>T_\theta$ ,} \\ \text{ we have $\varphi_1( g_tu(\theta) x) \geq t^{\gamma} $} } \}) &\leq s- \frac{1}{a_1 + b_1}\left(\eta^{(\delta)}+  \frac{\eta_1^{(\delta)} a_mb_n\omega}{a_m + b_n+ a_m\omega} \right).
\end{align*}
Now taking the limit as $\delta \rightarrow 0$, we get that proposition holds in this case as well.
\end{proof}

\begin{proof}[Proof of Theorem \ref{main thm 2}]
    By Proposition \ref{Critical Exponent is positive}, we know that $\zeta_l(\mu) > 0$ for all $1 \leq l \leq d-1$. Therefore, we can construct a sequence $\eta_1, \ldots, \eta_{d-1}$ such that:
    \begin{align}
        0<\eta_i &\leq \zeta_i(\mu) \quad \text{for all } 1 \leq i \leq d-1, \label{eq: 1 2 1} \\
        \frac{1}{\eta_{i-j}} + \frac{1}{\eta_{i+j}} &\leq \frac{2}{\eta_i} \quad \text{for all } 1 \leq i \leq d-1, \, 1 \leq j \leq \min\{i, d-i\}, \label{eq: 1 2 2}
    \end{align}
    where $1/\eta_0 = 1/\eta_d := 0$. For any such sequence, the results in \eqref{eq: main thm 2 1} and \eqref{eq: main thm 2 2} follow directly from Proposition \ref{prop:Genral Estimate} and fact that $\lambda_1(x) =\varphi_1(x)^{-1}$ for all $x \in \X$. This completes the proof of the first part of the theorem.

    To establish \eqref{eq: main thm 3}, observe that Lemma \ref{lem: Critical Exponent Real Case} guarantees that the constants defined in \eqref{eq: main thm 3} satisfy the conditions in \eqref{eq: 1 2 1} and \eqref{eq: 1 2 2}. Similarly, for the constants defined in \eqref{eq: main thm 4} (respectively, \eqref{eq: main thm 5}), the validity of \eqref{eq: 1 2 1} and \eqref{eq: 1 2 2} follows from Lemma \ref{lem: n 1 Critical Expo} (respectively, Lemma \ref{lem: m 1 Critical Expo}). 

    Thus, the theorem is proven.
\end{proof}

\begin{proof}[Proof of Theorem \ref{main thm}]
    Using Lemmas \ref{lem: Sing Dynamical Interpretation} and \ref{lem: omega sing dynamical Interpretation}, we have for all $\omega>0$
    \begin{align*}
        \Sing(a,b) &\subset \Divergent(\Z^d,1), \\
        \Sing(a,b,\omega) &\subset \bigcap_{\omega'< \omega} \{\theta \in \Mat: \text{ for all large $t$, we have } \lambda_1(g_tu(\theta)x)\leq t^{\frac{ a_mb_n\omega'}{a_m+b_n+ a_m\omega'}} \}.
    \end{align*}
    The theorem now follows from Theorem \ref{main thm 2}.
\end{proof}

\bibliography{Biblio}
\end{document}